\documentclass[11pt, a4paper,reqno]{amsart}

\usepackage{xspace,
	amsthm,
	amsmath,
	amssymb,
	bbm,
	tikz,
	graphicx,
	subfig,
	xcolor,
	keyval,
	mathtools
}

\usetikzlibrary{decorations.pathreplacing}

\usepackage[utf8]{inputenc}

\usepackage{geometry}
\usepackage{amsmath,stackrel}
\usepackage{tikz}
\usetikzlibrary{shapes.geometric,calc,decorations.markings,math}

\tikzstyle{nodo}=[circle,draw,fill,inner sep=0pt,minimum size=1.5mm]
\tikzstyle{infinito}=[circle,inner sep=0pt,minimum size=0mm]

\newcommand\rr{{\mathbb R}}

\newcommand\nn{{\mathbb N}}

\newcommand\f{\frac}

\DeclareMathOperator*{\esssup}{ess\,sup}

\newtheorem{theorem}{Theorem}[section]
\newtheorem{proposition}[theorem]{Proposition}
\newtheorem{lemma}[theorem]{Lemma}

\theoremstyle{remark}
\newtheorem{remark}{Remark}[section]
\newtheorem*{remark*}{Remark}

\theoremstyle{definition}
\newtheorem{definition}[theorem]{Definition}

\tikzset{every loop/.style={min distance=10mm,in=300,out=240,looseness=10}}

\tikzset{place/.style={circle,thick,draw=blue!75,fill=blue!20,minimum
size=6mm}}
\tikzset{place2/.style={circle,thick,draw=red!75,fill=red!20,minimum
size=6mm}}

\newcommand{\G}{\mathcal{G}}

\newcommand{\W}{\mathcal{W}}
\newcommand{\ee}{\varepsilon}

\newcommand{\weaks}{\stackrel{*}{\rightharpoonup}}

\numberwithin{equation}{section}

\begin{document}
	
	\title[Continuum limit of epidemiological models on graphs]{On the continuum limit of epidemiological models on graphs: Convergence and Approximation results}

	\author[B. Ayuso de Dios]{Blanca Ayuso de Dios$^{1,3}$}
	\address{$^1$Universit\'a degli Studi Milano--Bicocca, Dipartimento di Matematica e Applicazioni, via Roberto Cozzi, 55, 20126 Milano, Italy.}
	\email{blanca.ayuso@unimib.it}
	\author[S. Dovetta]{Simone Dovetta$^2$}
	\address{$^2$Politecnico di Torino, Dipartimento di Scienze Matematiche, Corso Duca degli Abruzzi 24, 10125 Torino, Italy.}
	\email{simone.dovetta@polito.it}
	\author[L.V. Spinolo]{Laura V. Spinolo$^3$}
	\address{$^3$IMATI-CNR, Via Ferrata 5, 27100 Pavia, Italy.}
	\email{spinolo@imati.cnr.it}
	
	
	\begin{abstract} 
      We focus on an epidemiological model (the archetypical  SIR system) defined on graphs and study the asymptotic behavior of the solutions as the number of vertices in the graph diverges. By relying on the theory of so called \emph{graphons} we provide a characterization of the limit and establish convergence results. We also provide approximation results for both deterministic and random discretizations. \\ 
      
	\end{abstract}
	
	\maketitle

	\vspace{-1cm}

	\section{Introduction}
	In the present work we consider epidemiological models (that is, models describing the spreading of an infectious disease) defined on graphs. 
	In real-world applications, each vertex of the graph can represent a single individual, but also a group of people living together or sharing 
	a specific trait (for instance, people in the same age group), or a geographical entity like a neighbourhood, a town, a region. We are 
	specifically interested in describing the behaviour of the model as the number $n$ of vertices in the graph goes to $+ \infty$ and we establish convergence and approximation results, based on both deterministic and random algorithms. 
	Our analysis borrows tools from the theory of so called \emph{graphons}, which plays a prominent role in modern graph theory. To the best of our knowledge, the present work is among the very first ones using the theory of graphons to study the asymptotic properties of epidemiological models on graphs as the number of vertices diverges (see also \cite{sis-graphon2,GC19,VFG20} for papers with a completely different focus).

	\smallskip
	Although the origin of mathematical epidemiology can be traced back to Bernoulli \cite{bernoulli}, nowadays the archetype of epidemiological models is the celebrated SIR (Susceptible-Infected-Recovered or Removed) system introduced by Kermack and McKendrick in \cite{McK26,KMK27}. The dimensionless form of the SIR system is 
        	\begin{equation*}
	         \f{d s}{dt} = - \beta s i\,, \qquad \f{d i}{dt} = \beta si - \gamma i\, ,\qquad \f{d r }{dt} = \gamma i\,.
	\end{equation*}
	Here, the unknowns $s, i, r: \rr_+ \to \rr$ represent the percentage of susceptible, infected and recovered individuals, whereas $\beta \ge 0$ and $\gamma\ge 0$ denote the (possibly time-dependent) infecting and recovering coefficients, respectively. Note furthermore that the sum $s+i+r\equiv 1$ remains constant in time. The analysis of SIR models, and of more refined variants (SEIR, MSEIR, ...), has grown enormously through the decades, resulting in a huge amount of literature (we refer e.g. to the monographs and lecture notes \cite{BCC01,BDW08,DG05,Mur02,Mur03}, to the review \cite{Het00} and to the references therein for extended discussions on this subject).
	Note that, in its standard version above, the SIR model implicitly assumes that contacts among individuals are uniform, i.e. that every individual has the same probability of interacting with any other. Since this is far from being true in real life, where social interactions are governed by complex patterns, several authors have investigated epidemiological models defined on networks (see for instance  \cite{BBPSV05,DGM08,MPSV02,New02,PSV01} and the reviews \cite{NPP16,PSCMV15}). The network approach is by now regarded as a feasible tool to describe epidemic spreads through heterogeneous populations. The importance of taking into account the role of population heterogeneity in the modelling of an epidemics has been stressed once more with the COVID19 pandemic, which fuelled in the last years an upsurge of research on epidemiological models from different mathematical communities (recent works in these directions are for instance \cite{perthame0, BorgsSIR, bellomo}).

\smallskip
	In what follows, even though we are confident that the results discussed in the present work could be applied to refined models with more compartments (SEIR, MSEIR,$\,\dots$), to ease the exposition we focus on the SIR model only. 
	To discuss our main contributions we now fix a graph $\G_n$ with $n$ vertices labeled $\{1, \dots, n\}$, and we consider the SIR model 
	\begin{equation}
		\label{SIR_G}
		\begin{cases}
		\displaystyle{\f{ds_j^n}{dt}=-s_j^n(t)\f1{n}\sum_{k=1}^{n}\beta_k^n(t)A_{jk}^n(t)i_k^n(t) }& \\
		\displaystyle{ \f{di_j^n}{dt}=s_j^n(t)\f1{n}\sum_{k=1}^{n}\beta_k^n(t)A_{jk}^n(t)i_k^n(t)-\gamma_j^n(t) i_j^n(t)} &\qquad\qquad\qquad j=1,\,\dots,\,n\\
		\displaystyle{ \f{dr_j^n}{dt}=\gamma_j^n(t) i_j^n(t)}\,. 
		\end{cases}
	\end{equation}
	In the above system, $s^n_j, i^n_j, r^n_j$ represent the percentage of susceptible, infected and recovered individuals at the vertex $j$, whereas $\beta^n_j$ and $\gamma^n_j$ are  the (possibly time dependent) infecting and recovering coefficients. $(A^n_{jk})_{j, k=1, \dots n}$ is the adjacency matrix  which describes the connections (\emph{edges}) between the vertices: the (possibly time dependent) coefficient $A^n_{jk}$ indicates the weight given to the edge between the vertices $j$ and $k$. If $j$ and $k$ are not connected, then $A^n_{jk}=0$. We consider graphs which contain no loops and no multiple edges, undirected (edges are not oriented) and with nonnegative weights. These characteristics of $\G_n$ translate into the conditions
	\begin{equation} \label{e:simpleun} 
	        A^n_{jj}=0 \; \text{for every $j$}, \qquad A^n_{jk}=A^n_{kj} \ge 0 \quad \text{for every $j, k=1, \dots, n$}.
	\end{equation}
	 Note however that in the present work we actually never use the no-loops condition $A^n_{jj}=0$.
System~\eqref{SIR_G} is augmented with the initial conditions 
	\begin{equation}
	\label{e:idSIR_G}
	        s_j^n(0)=s_{j,0}^n,\quad i_j^n(0)=i_{j,0}^n,\quad r_j^n(0)=r_{j,0}^n \qquad \text{for every $j=1,\,\dots,\,n$}
	\end{equation}
	and in view of modelling considerations they are such that
	\begin{equation}
	\label{e:idSIR_G2}
	        0 \leq s_{j,0}^n, i_{j,0}^n, r_{j,0}^n \leq 1,\qquad s_{j,0}^n+ i_{j,0}^n+r_{j,0}^n =1  \qquad \text{for every $j=1,\,\dots,\,n$,}
	\end{equation}
	which as we will show (see Lemma~\ref{l:zerouno} in \S\ref{sec:limit}) implies that the solution satisfies 
	\begin{equation}
	\label{e:01SIR_G}
	        0 \leq s_{j}^n(t), i_{j}^n(t), r_{j}^n(t) \leq 1,\quad s_{j}^n(t)+ i_{j}^n(t)+r_{j}^n(t) \equiv 1 \quad \text{for every $j=1,\,\dots,\,n$ and every $t \in \rr_+$}.
	\end{equation}
	In a nutshell, the present paper aims at discussing the $n\to + \infty$ limit of system~\eqref{SIR_G}. To do so, we rely on the celebrated theory of so called \emph{graphons}, which has recently flourished with a series of fundamental works like~\cite{BCCG21,BCCZ18,BCCZ19,BCL10,BCLSV06,BCLSV08,BCLSV12,BS02,Lov12,LS07}. In general, a function $W \in L^1 ([0, 1]^2;\mathbb{R})$ is called a \emph{graphon} if  $W(x, y) = W(y, x)$ (see Definition~\ref{d:graphon} in Appendix \ref{sec:graphons}). We refer to Appendix \ref{sec:graphons} for some rigorous definitions and a brief overview, mainly based on~\cite{BCCZ19}, of the main results concerning graphons related to the present paper.  Here we just mention that the basic idea of the theory of graphons is to identify a given graph $\G_n$ with $n$ vertices and adjacency matrix $(A_{jk})_{j, k=1, \dots, n}$ with a piecewise constant function $W_{\G_n}: [0, 1]^{2} \to \rr$ (the so called {\em step-graphon}) defined on the unit square. By considering a uniform partition of $]0,1[$ into $n$ intervals
\begin{equation} \label{e:ij}
    I^n_j : = \left] \f{j-1}{n}, \f{j}{n}\right[, \quad j=1, \dots, n,
\end{equation}	
and setting 
	\begin{equation}\label{G-->W_G}
       W_{\G_n}(x,y)=A^n_{jk}
       \quad \text{if $(x, y) \in I^n_j \times I^n_k$},
       \end{equation}
       for $\G_n$ satisfying~\eqref{e:simpleun}, the piecewise constant function $W_{\G_n}$ fullfills $W_{\G_n}(x, x) =0$, $W_{\G_n}(x, y) = W_{\G_n}(y, x)$ and $W_{\G_n} \ge 0$. 
       \begin{figure}[t]  
	\centering	\subfloat{\includegraphics[height=0.2\textwidth,width=0.2\textwidth]{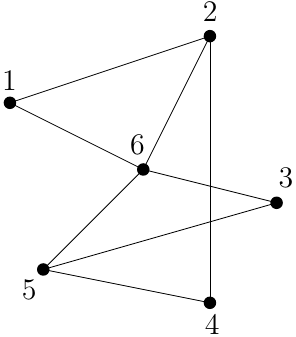}}\qquad
	\subfloat{\includegraphics[width=0.2\textwidth,height=0.2\textwidth]{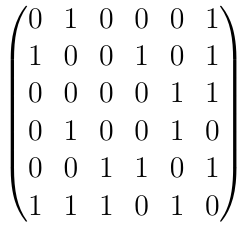}}\qquad
	\subfloat{\includegraphics[height=0.2\textwidth,width=0.2\textwidth]{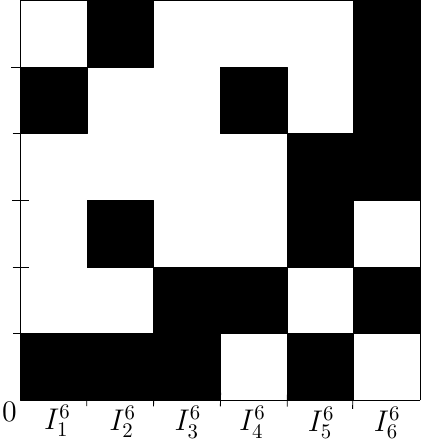}}
	\caption{A simple graph $\G$ with $n=6$ (leftmost), its adjacency matrix (center) and the pixel representation (rightmost) of the graphon $W_\G$ associated to $\G$ as in \eqref{G-->W_G}. }
	\label{f:fig1}
\end{figure}
In Figure \ref{f:fig1} is given the so called pixel picture of the graph. On the left there is a simple graph\footnote{Hence $A_{jk} \in \{0,1\}$ for all $j,k$} $\G_n$ with $n=6$, on the middle its adjacency matrix and on the right the  pixel picture of the piecewise constant step-graphon $W_{\G_n}$. Here, the little squares have all edge length $1/n$. The $0's$ entries in $A$ correspond to white squares, the $1's$ to black squares. The origin $(0,0)$ is placed at the upper left corner, in analogy with indexing matrix elements.

    To apply the theory of graphons to system~\eqref{SIR_G}, we need to identify functions defined on the set of vertices of $\G_n$ with functions defined on the unit interval. To this end, for a graph $\G_n$ with $n$ vertices we consider a vector valued function $(u^n_i)_{i=1, \dots,n}$ defined say on $\rr_+$, where each component $u^n_j$ represents the value attained at a given vertex $j$ of $\G_n$. In this work, we will always identify the vector valued function $(u^n_i)_{i=1, \dots,n}$ with the \emph{scalar} piecewise constant function $u^n$ defined on $\rr_+ \times [0, 1]$ given by 
\begin{equation}
	\label{uG->u01}
	u^n(t, x):=\sum_{j=1}^{n} u^n_j (t) \mathbbm{1}_{I^n_j}(x)\,,
	\end{equation}
	with $\mathbbm{1}_{I^n_j}$ denoting the characteristic function of the interval $I^n_j$, defined in~\eqref{e:ij}.  
	
	In this framework, the (formal) limit of~\eqref{SIR_G} is thus the system 
        \begin{equation}
		\label{SIR_W}
		\begin{cases}
		\partial_t s(t,x)=-s(t,x) \displaystyle{\int_0^1\beta(t,y)W(t,x,y)i(t,y)\,dy} &\\
		\partial_t i(t,x)=s(t,x)
         \displaystyle{\int_0^1\beta(t,y)W(t,x,y)i(t,y)\,dy-\gamma(t,x)i(t,x) } & \qquad (t, x) \in \rr_+ \times [0, 1] 
        \\
		\partial_tr(t,x)=\gamma(t,x)i(t,x)\,. \phantom{\displaystyle{\int}} & \\
		\end{cases}
	\end{equation}
	We augment~\eqref{SIR_W} with the initial condition 
	\begin{equation}
	\label{e:idSIR_W}
	   		s(0,x)=s_0(x),\; i(0,x)=i_0(x),\; r(0,x)=r_0(x)\, 
	\end{equation}
	and in virtue of~\eqref{e:01SIR_G} we assume
	\begin{equation}
	\label{e:01SIR_W}
	            0 \leq s_{0}, i_{0}, r_{0} \leq 1,\qquad s_{0}+ i_{0}+r_{0} =1  \qquad \text{a.e. on $[0, 1]$.}
	\end{equation}
The notion of solution for system \eqref{SIR_W} we will consider in this paper, which we refer to as {\em distributional solution}, is made precise in the next definition.
\begin{definition}\label{def:sol}
	Fix $T\in (0,\infty]$ and $W\in L^1_{\mathrm{loc}}\left([0,T] \times[0,1]^2;\mathbb{R}\right)$, $\beta,\gamma\in L^\infty([0,T]\times[0,1])$ and $s_0,i_0,r_0\in L^\infty(0,1)$. A triple $(s,i,r) \in \left(L^\infty\left([0,T] \times[0,1]\right)\right)^3$ is a distributional solution of the Cauchy problem \eqref{SIR_W}, \eqref{e:idSIR_W} if for every $(\varphi,  \psi, \eta) \in (C_c^\infty\left([0,T) \times[0,1]\right))^3$ we have 
	\begin{equation}
		\label{eq:weakLinfty}
		\begin{split}
			& 	
			\int_{\rr_+} \int_0^1s(t,x)\left(\partial_t\varphi(t,x)-\varphi(t,x)\int_0^1\beta(t,y)W(t,x,y)i(t,y)\,dy\right)dxdt =
			\int_0^1s_0(x)\varphi(0,x)\,dx \\ 	      
			&\int_{\rr_+}\int_0^1i(t,x)\left(\partial_t\psi(t,x)-\gamma(t,x)\psi(t,x)\right)+
			s(t,x)\psi(t,x)\int_0^1\beta(t,y)W(t,x,y)i(t,y)\,dydxdt\\
			& \qquad \qquad \qquad \qquad \qquad \qquad \qquad \qquad \qquad \qquad \qquad \qquad =\int_0^1i_0(x)\psi(0,x)\,dx  \\
			&
			\int_{\rr_+}\int_0^1r(t,x)\partial_t\eta(t,x)+\gamma(t,x)\eta(t,x)i(t,x)dxdt=
			\int_0^1r_0(x)\eta(0,x)\,dx.   \\
		\end{split}
	\end{equation}
\end{definition}

Our first result, Proposition \ref{p:uniex}, guarantees the existence and uniqueness of a distributional solution (in the sense of Definition \ref{def:sol}) to the Cauchy problem~\eqref{SIR_W}, \eqref{e:idSIR_W}.
\begin{proposition}
\label{p:uniex}
Fix $T>0$ and $W\in L^1\left([0, T] \times[0,1]^2;\mathbb{R}\right)$ satisfying $W \ge 0$ and $s_0,i_0,r_0\in L^\infty(0,1)$ satisfying~\eqref{e:01SIR_W}. 
Assume furthermore that $\gamma, \beta \in L^\infty ([0, T] \times [0, 1])$ satisfy $\gamma, \beta \ge 0$ almost everywhere on $[0,T]\times[0,1]$. 
Then if either of the following conditions hold:
\begin{description}
\item[$\bullet$ {\sc Case 1}] $W\in L^2\left([0, T] \times[0,1]^2;\mathbb{R}\right)$; or
\item[$\bullet$ {\sc Case 2}] there is a constant
$K_d>0 $ such that
\end{description}
\begin{equation}
	\label{bound_W}
	\esssup_{t \in [0, T], \, x \in [0,1]}\int_0^1W (t,x,y)\,dy\leq K_{d}\,;
	\end{equation}	
	
 then the Cauchy problem~\eqref{SIR_W}, \eqref{e:idSIR_W} admits exactly one distributional solution satisfying
\begin{equation}
\label{e:sirbound}
    0 \leq s, i, r \leq 1, \quad s+ i+r =1, \quad \text{a.e. on $[0, T] \times [0, 1]$}. 
\end{equation}
\end{proposition}
	
	Prior to state our first main result, we recall the meaning of the so called \emph{cut norm} $\| \cdot \|_{\Box}$, which for a graphon $W$ is defined as 
         		\begin{equation}
         			\label{cut1}
         			\|W\|_{\Box}:=\sup_{S,T\subseteq[0,1]}\left|\iint_{S\times T}W(x,y)\,dxdy\right|\,,
         		\end{equation}
         		where the supremum is taken over all pairs of measurable subsets $S,T$ in $[0,1]$.
         	This norm is the prototypical one in the theory of graphons (see the whole discussion in Appendix \ref{sec:graphons}).
	We can now state our first main convergence result.
         \begin{theorem}
     \label{THM 1} 
    Fix $T>0$  and let $\{ \G_n \}_{n \in \nn}$ be a sequence of time-dependent weighted undirected graphs with adjacency matrix satisfying~\eqref{e:simpleun} and 
     \begin{equation}
     \label{e:conv_tloc_cut}
      \lim_{n \to + \infty }\int_0^T\left\| W_{\G_n} (t, \cdot, \cdot)-W(t, \cdot, \cdot)\right\|_{\Box}\,dt= 0
 \end{equation}
      for some graphon $W \in L^1 ([0, T] \times [0, 1]^2;\mathbb{R})$. Assume furthermore that either of the following conditions holds:
      \begin{description}
\item[ $\bullet$ {\sc Case 1}]  there is a constant $K_0>0$ such that
      \begin{equation} 
     \label{e:bdl2}
       \| W_{\G_n}  \|_{L^2 ([0, T] \times [0, 1]^2;\mathbb{R})}  \leq K_{0}  \qquad  \text{for every $n \in \mathbb{N}$};\text{ or}
 \end{equation}
 \item[ $\bullet$ {\sc Case 2}]  there is a constant $K_1>0$ such that
      \begin{equation}
	\label{bound_Wn}
 	\esssup_{t \in [0, T], \, x \in [0,1]}\int_0^1W_{\G_n} (t,x,y)\,dy\leq K_{1} \qquad  \text{for every } n \in \nn.
	\end{equation}
	\end{description}
	
      Recalling the notation \eqref{uG->u01}, assume also that the coefficients $\beta^n,\gamma^n$ satisfy
      \begin{equation}
	\label{e:hpgammbeta}
	0 \leq  \beta^n,
	 \gamma^n  \leq  M \; \text{a.e. on $[0, T] \times [0, 1]$}, \qquad
	\beta^n \to \beta, \; \gamma^n \to \gamma \; 
	\text{strongly in $L^1([0, T] \times [0, 1])$}\,,
\end{equation}
for some constant $M>0$ and some functions $\beta, \gamma \in L^\infty ([0, T]\times [0, 1])$.

If $(s^n, i^n, r^n)$ is the solution of~\eqref{SIR_G},~\eqref{e:idSIR_G} for given initial data $s_0^n,i_0^n, r_0^n$ satisfying \eqref{e:idSIR_G2} for every $n$, then we have up to subsequences
\begin{equation}
\label{conv_C Lw}
(s^n,i^n,r^n) \weaks (s,i,r) \quad \text{weakly$^\ast$ in } 
L^\infty([0, T] \times [0,1];\rr^3)  \;\text{ as }n\to+\infty\,,
\end{equation}
where $(s,i,r): [0, T] \times[0,1]\to[0,1]^3$ is the distributional solution of \eqref{SIR_W}, \eqref{e:idSIR_W} satisfying~\eqref{e:sirbound} and attaining the initial condition~\eqref{e:idSIR_W} for some data $(s_0, i_0, r_0)$ satisfying~\eqref{e:01SIR_W}.
\end{theorem}
Some remarks are in order. 
First, assumption~\eqref{e:conv_tloc_cut}  is entirely natural from the point of view of graphon theory in view of the analysis in~\cite{BCCZ19,BCLSV08}. We refer  to Appendix \ref{sec:graphons} and in particular to Remark~\ref{r:discussion} for a more technical discussion, here we just mention that assumption~\eqref{e:conv_tloc_cut} is satisfied (up to subsequences and vertices relabelling) in two very relevant cases: (i) if $\{ \G_n \}_{n \in \nn}$ is a dense\footnote{We refer to Appendix \ref{sec:graphons} for the definitions of dense and sparse sequence of graphs} sequence of graphs such that $\| W_{\G_n} \|_{L^p([0,T]\times[0,1]^2;\mathbb{R})}$ is uniformly bounded for some $p>1$; (ii) if  $\{ \G_n \}_{n \in \nn}$ is a sparse sequence of graphs satisfying suitable topological assumptions (see Appendix \ref{sec:graphons} for the precise conditions), then~\eqref{e:conv_tloc_cut}  holds true up to renormalization\footnote{Actually, condition~\eqref{e:conv_tloc_cut}  is directly satisfied with no renormalization by sparse graph sequences too, but this is of limited interest since the limit would be the trivial function $W\equiv0$}, i.e. provided we replace $\G_n$ by $\G_n / \| W_{\G_n} \|_{L^1([0,T]\times[0,1]^2;\mathbb{R})}$.

  
Second, both conditions~\eqref{e:bdl2} and~\eqref{bound_Wn} pass to the limit, i.e. they are satisfied by the limit graphon $W$ in~\eqref{e:conv_tloc_cut}, see respectively Lemmas~\ref{l:stimal2}-\ref{l:stimabound} below. This implies that the limit system~\eqref{SIR_W} satisfies the assumptions of Proposition~\ref{p:uniex} and in particular that there is a unique solution of the Cauchy problem~\eqref{SIR_W},~\eqref{e:idSIR_W} satisfying~\eqref{e:sirbound}. This in turn ensures that, by fixing the initial data~\eqref{e:idSIR_W} of the limit problem,  then there is no need to pass to subsequences in~\eqref{conv_C Lw}, as the whole sequence converges.

Third, by looking at the proof of Theorem~\ref{THM 1} one realizes that, if $\G_n$ is a sequence of random graphs satisfying~\eqref{e:conv_tloc_cut} and either~\eqref{e:bdl2} or~\eqref{bound_Wn} almost surely, 
then~\eqref{conv_C Lw} holds almost surely.

Finally, Theorem~\ref{THM 1} can be regarded as a \emph{convergence} result: given a sequence of graphs $\{\G_n\}_{n\in\mathbb{N}}$ satisfying~\eqref{e:conv_tloc_cut} for some $W$, we characterize the asymptotic behaviour of the corresponding sequence of SIR models~\eqref{SIR_G}.

In the next results, we somehow adopt the opposite viewpoint: given a graphon $W\in L^1([0, T] \times [0,1]^2;\mathbb{R})$ satisfying suitable conditions and the corresponding system~\eqref{SIR_W}, we look at discrete approximations to  the solution of~\eqref{SIR_W}. Our approach here focuses on the approximation of $W$ by sequences of graphs constructed in such a way that the corresponding solutions of~\eqref{SIR_G} converge to the solution of~\eqref{SIR_W}. 
Obviously, different samplings of $W$ provide different approximating graphs and the actual convergence of the specific approximation to the limit problem depends on the considered discretization.  In this work, we study both deterministic and random approximations. However, before discussing in detail our results in this direction, we stress immediately that all the sampling procedures we take into account retain much more information on the limit graphon $W$ than the rather mild assumption \eqref{e:conv_tloc_cut} required by Theorem \ref{THM 1}. As a consequence, all the convergence results we obtain (see e.g. \eqref{e:cisiamo}-\eqref{e:c0l1} and \eqref{e:convP}-\eqref{e:convP2}  below) hold in a much stronger topology than the one in \eqref{conv_C Lw}.

We now introduce the deterministic approximation we consider. Let $n \in \mathbb N$, recall~\eqref{e:ij} 
and set
\begin{equation}
\label{e:average}
\begin{split}
     & < u>_{I^n_j} : =  n \int_{I_j^n} u(x) dx \quad \text{for every }j=1, \dots, n,\qquad\forall u \in L^1(0, 1), \\ 
     & < U>_{I^n_j\times I^n_k} : =  n^2 \iint_{I_j ^n\times I_k^n} U(x) dx  \quad \text{for every }j, k=1, \dots, n,\qquad\forall U \in L^1([0, 1]^2). 
      \end{split}
\end{equation}
Given a possibly time-dependent graphon $W$,  for every $n \in \mathbb{N}$ we then let $\G_n$ be the  possibly time-dependent graph with $n$ vertices and adjacency matrix 
\begin{equation} \label{e:adj}
    A^n_{jk} (t): =  < W(t, \cdot, \cdot)>_{I_j^n\times I_k^n}. 
\end{equation}
We can now state our first approximation result.

\begin{theorem}
		\label{thm:det} 
		Fix $T>0$ and assume $W \in L^1([0, T] \times [0,1]^2; \mathbb{R}_+)$ is a non-negative time-dependent graphon. Let $\beta,\gamma\in L^\infty\left([0, T]\times[0,1]\right)$ and $s_0,i_0,r_0\in L^\infty(0,1)$ satisfy $\beta, \gamma \ge 0$ 
	and~\eqref{e:01SIR_W}, respectively. 
	Set 
		\begin{equation}
		 \label{e:betagamma}
		 \begin{split}
		\beta^{n}_j(t):= < \beta(t, \cdot)>_{I_j^n} , \quad  \gamma^{n}_j(t): =< \gamma (t, \cdot)>_{I_j^n} \quad \text{for a.e. $t \in \rr_+$ and every $j=1, \dots, n$}    \\
		s^n_{j,0} : = < s_0>_{I^n_j}, \quad  i^n_{j,0} : = < i_0>_{I^n_j}, \quad r^n_{j,0} : = < r_0>_{I^n_j} 
		 \quad \text{for every $j=1, \dots, n$}
		\end{split}
		\end{equation}
		and let $\G_n$ be the graph with adjacency matrix $(A^n_{jk})_{j, k=1, \dots, n}$ defined in~\eqref{e:adj}. Let $(s^n, i^n, r^n)$ be the solution of the Cauchy problem~\eqref{SIR_G}, \eqref{e:idSIR_G} written using the notation~\eqref{uG->u01}. Then,  provided $(s, i, r)$ is the distributional solution of \eqref{SIR_W}, \eqref{e:idSIR_W} satisfying~\eqref{e:sirbound}, the following results hold:
		\begin{description}
		\item[$\bullet$ {\sc Case 1}] if $W \in L^2 ([0, T] \times [0, 1]^2; \mathbb{R}_+)$ then 
				\begin{equation}
\label{e:cisiamo}
    \lim_{n \to + \infty} \esssup_{t \in [0, T]}\left(
      \|s (t, \cdot) -  s^n (t, \cdot)  \|_{L^2(0, 1)} 
      +  \|i (t, \cdot) -  i^n (t, \cdot)  \|_{L^2(0, 1)} + 
     \|r (t, \cdot) - r^n (t, \cdot)  \|_{L^2(0, 1)} \right)=0;
\end{equation}
        
        \item[$\bullet$ {\sc Case 2}] if  $W$ satisfies ~\eqref{bound_W} then
         \begin{equation} \label{e:c0l1}
		 \lim_{n \to + \infty} \esssup_{t \in [0, T]}\left(
      \|s (t, \cdot) -  s^n (t, \cdot)  \|_{L^1(0, 1)} 
      +  \|i (t, \cdot) -  i^n (t, \cdot)  \|_{L^1(0, 1)} + 
     \|r (t, \cdot) - r^n (t, \cdot)  \|_{L^1(0, 1)} \right)=0.
		 		\end{equation}
				 \end{description}
				 
        \end{theorem}
Note that with the sampling procedure defined by~\eqref{e:adj} and~\eqref{e:betagamma} the resulting discretization  coincides with a discontinuous piecewise constant Galerkin approximation to \eqref{SIR_W}. In particular, if for instance $\beta^n_j$ is as in~\eqref{e:betagamma} and we define $\beta^n$ using the notation~\eqref{uG->u01}, then $\beta^n$ is exactly the $L^2$-orthogonal projection of $\beta$ onto the space of  piecewise constant functions on the uniform partition of $[0,1]$ of granularity $1/n$. Similarly, if $W(t, \cdot) \in L^2 ([0, 1]^2)$ the matrix $(A_{jk}^n(t))_{1\leq j,k\leq n}$ in~\eqref{e:adj} coincides with the $L^{2}$-orthogonal projection of the graphon $W(t, \cdot)$ onto the space of piecewise constant functions on the uniform cartesian grid of $[0,1]^{2}$. 

We now introduce our second sampling procedure, which belongs to the family of the so called $W$-random graphs introduced  in \cite{diaconisFreedman81}, popularized by \cite{LS} and then widely used in the context of graphon theory. 
Inspired by the construction given in~\cite{Med22,Med19}, for any graphon $W \in L^1 ([0, 1]^2;\mathbb{R}_+)$  and for $\alpha \in\, ]0,1[$ fixed, we introduce  the {\it scaled sparse $W$-random graph} $n^{\alpha}\mathcal{G}(n , W, n^{-\alpha})$ as the graph with $n$ nodes and adjacency matrix given by
\begin{equation}
        \label{e:adjr1} \begin{split}
              \mathbb P( A^n_{jk} = n^\alpha  ) = < \widehat W^n  >_{I^n_j \times I^n_k}, \; \text{where $\widehat W^n : = \min\{1, n^{-\alpha} W \}$}\,,\qquad 
             \mathbb P( A^n_{jk} = 0 ) =1 -  \mathbb P( A^n_{jk} = n^\alpha )\,.  
              \end{split} 
        \end{equation}
        The scaling factor $n^{-\alpha}$ is the so called target edge density of the graph and it does enter in the definition \eqref{e:adjr1} of the matrix probability distribution. On the one hand, it dictates the level to trim the unbounded graphon $W$. On the other hand, it allows to ensure that the cut-off in the minimum acts only on part of the edges.
        
       In the case of bounded graphons $W \in L^{\infty} ([0, 1]^2;[0,1])$ it is standard to choose $\alpha =0$, so that the adjacency matrix of the resulting graph boils down to a random matrix with entries following the Bernoulli distribution
	\begin{equation}\label{ran:infty}
		\mathbb P( A^n_{jk} = 1  ) =  < W  >_{I^n_j \times I^n_k}\,, \; \, \qquad 
		\mathbb P( A^n_{jk} = 0 ) =1 - < W  >_{I^n_j \times I^n_k}\,.
	\end{equation}
	Such sparse $W$-random graph is denoted here by $\mathcal{G}(n , W)$.  

\begin{remark}\label{remark:bla0}
The construction considered  in~\cite{Med22, Med19} (and more generally the definition of the so-called {\it sparse $W$-random graph} in the theory of graphons) is slightly different from \eqref{e:adjr1}.   In~\cite{Med22, Med19} the {\it sparse $W$-random graph} $\mathcal{G}(n,W,n^{-\alpha})$ has adjacency matrix  
	\begin{equation}\label{e:revisione}
	     \mathbb{P}(A^n_{jk} =1) =  < \widehat W^n  >_{I^n_j \times I^n_k}\,, \; \text{where $\widehat W^n : = \min\{1, n^{-\alpha} W \}$}, \qquad 
    \mathbb{P}(A^n_{jk} =0) = 1 -  < \widehat W^n  >_{I^n_j \times I^n_k}. 
	\end{equation}	 
Here, we consider directly the {\it scaled} sparse $W$-random graph $n^{\alpha}\mathcal{G}(n,W,n^{-\alpha})$ as defined in \eqref{ran:infty} since that is the graph capable of approximating the graphon $W$. 
\end{remark}

We now state the main results for the random approximation to \eqref{SIR_W} given by the scaled sparse $W$-random graphs $n^\alpha\mathcal{G}(n,W,n^{-\alpha})$ as random samplings of $W$. 
 \begin{theorem}
      \label{thm:ran} 
      Fix $T>0$ and let $W \in L^1([0,1]^2;\mathbb{R}_+)$ be a graphon and $\gamma, \beta \in L^\infty([0, T] \times [0, 1])$ satisfy $\beta, \gamma \ge 0$.  Let $\alpha \in\, ]0,1[$ be fixed. For every $n \in \mathbb N$ and $j=1, \dots, n$, let $\beta^n_j$, $\gamma^n_j$, $s^n_{j,0}, i^n_{j,0}, r^n_{j,0}$ be as in \eqref{e:betagamma} and $\G_n=n^{\alpha}\mathcal{G}(n,W,n^{-\alpha})$  be the scaled sparse $W$-random graph given by ~\eqref{e:adjr1}.  Let $(s^n, i^n, r^n)$ be the solution of the Cauchy problem~\eqref{SIR_G}, \eqref{e:idSIR_G} written by using the notation~\eqref{uG->u01}. Then the following results hold:
      \begin{description}
      \item[$\bullet$ {\sc Case 1}] if $W \in L^2 ([0, 1]^2; \mathbb{R}_+)$ and $\alpha \in\, ]0,1/2[$ then 
    \begin{equation}\begin{split}
\label{e:convP}
    \lim_{n \to + \infty}  \esssup_{t \in [0, T]} & \left( 
      \|s (t, \cdot) -  s^n (t, \cdot)  \|_{L^2(0, 1)} 
      +  \|i (t, \cdot) -  i^n (t, \cdot)  \|_{L^2(0, 1)}  \right. \\
      & \left. \qquad + 
     \|r (t, \cdot) -  r^n (t, \cdot)  \|_{L^2(0, 1)}  \right) =0
     \qquad \text{almost surely};
     \end{split}
\end{equation}
 \item[$\bullet$ {\sc Case 2}] if there is a constant $K_a>0$ such that 
         \begin{equation}
        \label{bound_W2}
			  \esssup_{x \in [0,1]}\int_0^1W (x,y)\,dy\leq K_a\;,
		\end{equation} 
		 then, for $\alpha \in\, ]0,1[$,
 \begin{equation}\begin{split}
\label{e:convP2}
   \lim_{n \to + \infty} \esssup_{t \in [0, T]} & \left(
      \|s (t, \cdot) -  s^n (t, \cdot)  \|_{L^1(0, 1)} 
      +  \|i (t, \cdot) -  i^n (t, \cdot)  \|_{L^1(0, 1)} \right. \\
      & \qquad \left. + 
     \|r (t, \cdot) - r^n (t, \cdot)  \|_{L^1(0, 1)} \right) =0
   \quad 
  \text{almost surely}.  \end{split}
\end{equation}
 \end{description}
\end{theorem} 
We wish to note (as will be evident from the proof of  Theorem~\ref{thm:ran}) that  if  $W \in L^{\infty} ([0, 1]^2;[0,1])$, the above theorem applies with  $\G_n $ given by the sparse random graph $\mathcal{G}(n , W)$ as defined in \eqref{ran:infty}. 
 
Prior to close this introduction we comment on the main novelties of our approximation results. The proof of Theorems~\ref{thm:det}--\ref{thm:ran}  borrow some of the techniques from the analysis in~\cite{Med19} concerning the limit of Kuramoto systems and nonlinear/nonlocal diffusion defined on graphs (see also~\cite{Med14a,Med14b,Med14CORR} and~\cite{BCD20} for a different but related problem). However, with respect to previous analyses \cite{KVM17, Med19, Med20a,Med22}, the specific features of the SIR system allows here to obtain the following improvements:
\begin{itemize}
\item[(i)] we relax the regularity assumptions on $W$, namely we require only $W \in L^2 ([0, 1]^2;\mathbb{R}_+)$ (cf. {\sc Case 1}). This in particular implies that Theorems \ref{thm:det}--\ref{thm:ran}  (contrary to ~\cite[Theorem 3.1]{Med19} and the main results in \cite{Med20a,Med22}) directly apply to the power law case $W(x, y) = (xy)^{-\mu}$, $\mu \in\, ] 0, 2^{-1}[$. As pointed out for instance in the introduction to~\cite{BCCG21, BCCZ19} (see also the one in  \cite{Med20a}), the power law case is usually considered technically challenging but also fascinating and compelling in view of applications, as several real-world networks have a power law structure. 
From the technical standpoint, when $W$ is an $L^2$ graphon we circumvent condition~\eqref{bound_W} by  a careful manipulation of the particular structure of the SIR system. 
Furthermore, unlike in~\cite{Med19,Med20a,Med22}, we can also go below the $L^2$ summability framework (cf.{\sc Case 2}), provided we adopt~\eqref{bound_W} (resp.  ~\eqref{bound_W2} in the random case). Note that~\eqref{bound_W} (resp.  ~\eqref{bound_W2} in the random case) is for instance satisfied by the graphons $W(x, y) = (x+y)^{-\mu}$, $\mu \in ]0, 1[$, which do not belong to $L^2([0,1]^2;\mathbb{R})$ if $\mu \ge 2^{-1}$;

\item[(ii)] for the deterministic approximation in Theorem~\ref{thm:det}, we allow for time-dependent graphons. This might be particularly relevant in view of applications because it allows to consider changes in the network structure (e.g. due to the implementation of lockdown measures).
However, with the random approximation in Theorem~\ref{thm:ran} our results are valid only for time independent graphons. This limitation is technical, mainly due to measurability issues on the integrals with respect to time of certain stochastic processes arising in the computations. Hence, at present we are not able to adapt the techniques used in the proof of Theorem~\ref{thm:ran} to deal with time dependent graphons $W$.
\end{itemize}

\subsection*{Paper outline}
The remainder of the paper is organized as follows. In \S\ref{sec:limit} we collect some preliminary  results. In \S\ref{sec:cutconv} we discuss the proof of Proposition~\ref{p:uniex} and of Theorem \ref{THM 1}.  \S\ref{s:det} contains  the proof of Theorem~\ref{thm:det} for the deterministic approximation, while in \S\ref{sec:proof2} we present the proof of Theorem~\ref{thm:ran} for the random approximation. The paper is completed with two appendices.  Appendix \ref{sec:notation} collects the main notation used in the paper, whereas Appendix \ref{sec:graphons} comprises a brief overview of some definitions and results on graphons (providing in particular sufficient conditions for \eqref{e:conv_tloc_cut} to hold).

\subsection*{Norms notation.} In the remainder of the paper, symbols like $\|u\|_p$ will be used to denote the $L^p$ norm of $u$ when computed with respect to all the variables $u$ depends on. Conversely, when the $L^p$ norm of $u$ is computed only with respect to some of its variables, we will use the full notation indicating explicitly the space of integration.

\section{Preliminary results}
\label{sec:limit}
In this section we collect some subsidiary results that will be useful  for the subsequent analysis. Precisely, in \S\ref{subsec:21} we prove existence and uniqueness of the solution for the discrete SIR system \eqref{SIR_G}, whereas in \S\ref{subsec:22} we establish a priori estimates for distributional solutions of the continuum SIR system \eqref{SIR_W}. In \S\ref{subsec:23} we show how conditions \eqref{e:bdl2} and \eqref{bound_Wn} for a sequence of graphons $W_n$ imply corresponding conditions on the limit graphon $W$. In \S\ref{subsec:24} we recall a standard convergence results for integral means that will be frequently used in the following.

\subsection{The discrete SIR model on graphs}
\label{subsec:21}
The following lemma is a direct consequence of the Cauchy Lipschitz Picard Lindel\"of Theorem on existence and uniqueness of solutions of ODEs. 
In the statement we use the notation~\eqref{uG->u01}.  
\begin{lemma}
	\label{lem:wpos_Gn}
	Fix $n \in \nn$  and assume $\beta^n_j, \gamma^n_j$, $A^n_{jk} \in L^\infty (\rr_+)$ satisfy $\beta^n_j, \gamma^n_j , A^n_{jk} \ge 0$ a.e. on $\rr_+$, for every $j, k=1, \dots, n$. For every given initial condition~\eqref{e:idSIR_G} satisfying~\eqref{e:idSIR_G2} there are $\nu>0$ and a 
	 unique solution ${s}^n, i^n, {r}^n: [0, \nu]  \times [0, 1]\to \rr$ of the Cauchy problem~\eqref{SIR_G}, \eqref{e:idSIR_G}.
\end{lemma}
Note that the above lemma only provides local-in-time existence and uniqueness. However, the local-in-time solution can be extended for every $t \ge 0$ and hence it is actually a global in-time-solution ${s}^n, i^n, {r}^n: \rr_+  \times [0, 1]\to \rr$. This is a consequence of the following lemma. The proof is standard, but we provide it for the sake of completeness.
\begin{lemma}
\label{l:zerouno}
Under the assumptions of Lemma~\ref{lem:wpos_Gn}, the solution of the Cauchy problem~\eqref{SIR_G}, \eqref{e:idSIR_G} satisfies~\eqref{e:01SIR_G}. 
\end{lemma} 

\begin{proof}
The equality in~\eqref{e:01SIR_G} follows from the fact that by adding the three lines of~\eqref{SIR_G} we get that the derivative of $ s^n_{j} (t)+ i^n_{j}  (t) + r^n_{j} (t)$ vanishes. 

Next, we point out that $s^n_j(t) =0$ for every $t \in \mathbb R$ is a solution of the equation at the first line of~\eqref{SIR_G}. By the uniqueness part of the Cauchy Lipschitz Picard Lindel\"of Theorem, this implies that, since $s^n_{j,0} \ge 0$, then $s^n_j(t) \ge 0$ for every $t$. 
We now prove that $i^n_j (t) \ge 0$, for every $j=1, \dots, n$. As a first step, we establish the proof under the further assumption
\begin{equation} \label{e:anjkmz}
A^n_{jk}>0, \quad \text{for every $j, k=1, \dots, n.$}
\end{equation}
We separately consider the two possible instances:

\smallskip
{\sc Instance (a):} $i^n_{j, 0}>0$ for every $j=1, \dots, n$.  We set   
$$
   \bar t : = \sup  \{ t \in \mathbb R_+ \! \!: i_j^n (\tau) >0 \; \text{for every} \, j=1, \dots, n \; \text{and every $\tau \in [0, t[$} \}.
$$
We point out that $\bar t >0$ and we now show that $\bar t = + \infty$. Assume by contradiction that $\bar t < + \infty$ 
and, just to fix the ideas, assume that $i_1^n (\bar t) =0$. Note that, in this case, it must be $s_1^n(\bar t)>0$. Indeed, if it were $s_1^n(\bar t)=i_1^n(\bar t)=0$, then $r_1^n(\bar t)=1$. However, $(0,0,1)$ is an equilibrium for the component $(s_1^n,i_1^n,r_1^n)$ (as for any other component), so that it would follow $s_{1,0}^n=i_{1,0}^n=0$, $r_{1,0}^n=1$, contradicting the assumption $i_{1,0}^n>0$. Hence, $s_1^n(\bar t)>0$.  Then there are two possibilities:
\begin{itemize}
\item[(i)] $i^n_2 (\bar t) = \dots = i^n_n (\bar t)=0$. Since $(s^n_1, \dots, s^n_n, 0, \dots, 0, r^n_1, \dots, r^n_n)$ is an equilibrium for~\eqref{SIR_G} for every $(s^n_1, \dots, s^n_n, r^n_1, \dots, r^n_n)$, this implies $i^n_1 (0) = \dots = i^n_n (0)=0$, which contradicts the definition of {\sc instance (a)}. 
\item[(ii)] there is $k$ such that $i^n_k (\bar t) >0$. Since  $A^n_{jk} > 0$ and $s_1^n(\bar t)>0$, this implies that $d i^n_1/dt >0$ at $t= \bar t$ and again contradicts the definition of $\bar t$. 
\end{itemize}

\smallskip
{\sc Instance (b):} there is $j=1, \dots, n$ such that $i^n_{j, 0}=0$. Note that, if $s_{j,0}^n=i_{j,0}^n=0$, then $r_{j,0}^n=1$, $(s_j^n(t),i_j^n(t),r_j^n(t))=(0,0,1)$ for every $t$ and the $j$-th component does not interact with any other component in the system. Therefore, without loss of generality we can restrict our attention only to the components $j$ for which $s_{j,0}^n>0$. Just to fix the ideas, we assume $i^n_{1, 0} =0$. If $i^n_{2, 0}= \dots = i^n_{n, 0}=0$ then by arguing as in item (i) above we conclude that $i^n_1 (t) = \dots = i^n_n (t)=0$ for every $t$, which in particular yields our claim. If there is $k$ such that $i^n_{k,0} >0$, then  $d i^n_j /dt >0$ at $t= 0$ for every $j$ for which $i_{j,0}^n=0$ (since $A_{jk}^n>0$ and $s_{j,0}^n>0$ by assumption). This implies that $i^n_{j} (t)>0$ for every $j=1, \dots, n$ and  $t \in\, ]0, \sigma[$ for some sufficiently small $\sigma>0$. We can then consider the Cauchy problem obtained by coupling~\eqref{SIR_G} with the datum assigned at $\sigma/2$ and apply the same argument as in the previous case. 

To conclude the proof of the inequality $i^n_j \ge 0$ we are left to remove the assumption~\eqref{e:anjkmz}. To this end, we rely on an approximation argument: we fix $\ee >0$, replace $A^n_{jk}$ with $A^n_{jk}+ \varepsilon$, for every $j, k=1, \dots, n$, and term $s^{n \ee}_j(t), i^{n \ee}_j (t), r^{n \ee}_{j} (t)$ the corresponding solution of the Cauchy problem~\eqref{SIR_G}, \eqref{e:idSIR_G}. Since $A^n_{jk}+ \varepsilon>0$, by the previous step $i^{n \ee}_j (t) \ge0$ for every $t$ and every $\ee>0$. By the continuous dependence of the solution of ODEs on parameters, $i^{n \ee}_j (t)$ converges to $i^n_j (t)$ as $\ee \to 0^+$ and  this implies $i^n_j (t)\ge 0$. 

Finally, since for every $j=1, \dots, n$ we have $\gamma^n_j \ge 0$ and $i^n_j \ge 0$, then $d r^n_j/dt \ge0$ and this yields $r^n_j \ge 0$.
\end{proof}

\subsection{A priori estimates for distributional solutions of~\eqref{SIR_W}}
\label{subsec:22}
We now derive a preliminary result which will be useful in the next section to establish Proposition~\ref{p:uniex}. 
\begin{lemma} \label{l:reg}
Fix $T>0$, $W\in L^1\left([0, T] \times[0,1]^2;\mathbb{R}\right)$ and $\beta, \gamma \in L^\infty ([0, T] \times [0, 1])$.
Let $(s, i, r) \in L^\infty ([0,T] \times [0, 1]; \mathbb{R}^3)$ be a distributional solution of~\eqref{SIR_W} satisfying~\eqref{e:sirbound}. Then, either:\\

$\bullet$ {\sc Case 1}: if $W\in L^2 \left([0, T] \times[0,1]^2;\mathbb{R}\right)$, then $\partial_t s, \partial_t i, \partial_t r \in 
 L^2 \left([0, T] \times[0,1] \right)$ and 
 \begin{equation}
 \label{e:boundder}
          \| \partial_t s \|_{2},    \| \partial_t i \|_{2},    \| \partial_t r \|_{2}
          \leq C( \| W \|_{2}, \| \beta \|_{\infty}, \| \gamma \|_{\infty}),
           \end{equation}
          where $C( \| W \|_{2}, \| \beta \|_{\infty}, \| \gamma \|_{\infty})$ is a suitable constant only depending on $\| W \|_{2}, \| \beta \|_{\infty}$ and $\| \gamma \|_{\infty}$; or
          
$\bullet$ {\sc Case 2}: if~\eqref{bound_W} holds true, then $\partial_t s, \partial_t i, \partial_t r \in 
 L^\infty \left([0, T] \times[0,1] \right)$  and        
\begin{equation}
 \label{e:boundder2}
          \| \partial_t s \|_{\infty} ,    \| \partial_t i \|_{\infty} ,    \| \partial_t r \|_{\infty}
          \leq C( K_d, \| \beta \|_{\infty}, \| \gamma \|_{\infty})\,,
           \end{equation}
  where $C( K_d, \| \beta \|_{\infty}, \| \gamma \|_{\infty})$ is a suitable constant only depending on $K_d, \| \beta \|_{\infty}$ and $\| \gamma \|_{\infty}$.          
\end{lemma}
\begin{proof}
Let $(s,i, r)$ be a distributional solution of~\eqref{SIR_W} satisfying~\eqref{e:sirbound}. Owing to~\eqref{eq:weakLinfty}, the distributional derivatives $(\partial_t s, \partial_t i, \partial_t r)$ are given by the right hand side of~\eqref{SIR_W} and hence are locally summable functions.
We now separately consider the following two cases.\\
{\sc Case 1:} $W\in L^2 \left([0, T] \times[0,1]^2;\mathbb{R}\right)$. 
By using Jensen's inequality we get  
\begin{equation*}
\begin{split}
      & \int_0^T \int_0^1  [ \partial_t s (t, x)]^2 dx dt    =  \int_0^T \int_0^1 \left( -s(t,x)
     \displaystyle{\int_0^1\beta(t,y)W(t,x,y)i(t,y)\,dy} \right)^2 dx dt \\&
     \qquad \qquad  \leq \| s \|^2_{\infty} \|\beta \|_{\infty}^2 \| i \|^2_{\infty} 
      \int_0^T  \int_0^1 \left(
     \displaystyle{\int_0^1W(t,x,y) dy} \right)^2 dx  dt \\ &\stackrel{\text{Jensen}}{\leq} 
     \| s \|^2_{\infty} \|\beta \|_{\infty}^2 \| i \|^2_{\infty} 
      \int_0^T  \int_0^1 \int_0^1W^2(t,x,y) dy dx dt \stackrel{\eqref{e:sirbound}}{\leq} 
        \|\beta \|_{\infty}^2 
       \int_0^T \int_0^1 \int_0^1W^2(t,x,y) dy dx dt\,,
\end{split}
\end{equation*}
which yields a control on the first term in~\eqref{e:boundder}, and relying on a similar argument we control also the other terms in~\eqref{e:boundder}. \\
{\sc Case 2:} $W$ satisfies \eqref{bound_W}. We get  
\begin{equation*}
      |\partial_t s(t, x) |  = \left| - s(t,x) \int_0^1\beta(t,y)W(t,x,y)i(t,y)\,dy\right|  \stackrel{\eqref{e:sirbound}}{\leq}  
         \|\beta\|_{\infty}  \left| \int_0^1W(t,x,y) \,dy\right| \stackrel{\eqref{bound_W}}{\leq}  
         K_d    \|\beta\|_{\infty}. 
\end{equation*}
By using an analogous argument we control the other terms in~\eqref{e:boundder2}. 
\end{proof}
\subsection{Limit conditions coming from~\eqref{e:bdl2} and~\eqref{bound_Wn}}
\label{subsec:23}
\begin{lemma}\label{l:stimal2}
{\sc Case 1.} Assume~\eqref{e:conv_tloc_cut} and~\eqref{e:bdl2}.  Then the limit graphon satisfies $\| W \|_{ L^2 ([0, T] \times [0, 1]^2;\mathbb{R})} \leq K_0$.
\end{lemma}
\begin{proof}
Owing to~\eqref{e:bdl2}, we have that, as $n\to+\infty$ and up to subsequences, $W_{\G_n} \rightharpoonup V$ weakly in $L^2([0, T] \times [0, 1]^2)$, for some limit function $V \in  L^2([0, T] \times [0, 1]^2)$ with $\|V\|_2\leq K_0$ by weak lower semicontinuity. To conclude, we have to show that $V \equiv W$. 
By the uniqueness of the distributional limit, it suffices to show that 
\begin{equation}\label{limitemis2}
    \lim_{n \to + \infty} \int_0^T  \iint_{[0, 1]^2}  \xi(t, x, y) [W_{\G_n} - W] (t, x, y) dx dy dt=0
\end{equation}
for every test function $\xi \in C^\infty_c (]0, T[ \times ]0, 1[^2)$. By the Stone-Weierstrass Theorem, to verify~\eqref{limitemis2} it suffices to show that for every $\eta~\in~C^\infty ([0, T])$, $\varphi, \psi \in C^\infty ([0, 1])$ we have 
\begin{equation}\label{limitemis}
    \lim_{n \to + \infty} \int_0^T \eta (t) \iint_{[0, 1]^2} \varphi(x) \psi(y) [W_{\G_n} - W] (t, x, y) dx dy dt=0. 
\end{equation}
To this end, we point out that 
\begin{equation*}
\begin{split}
        & \left| \! \int_0^T\!  \eta (t) \iint_{[0, 1]^2} \! \! \! \varphi(x) \psi(y) [W_{\G_n} - W] (t, x, y) dx dydt \right|  \leq 
          \int_0^T \! \left| \eta (t)\!  \iint_{[0, 1]^2} \! \! \! \varphi(x) \psi(y) [W_{\G_n} - W] (t, x, y) dx dy \right|  dt\\& \leq 
          \| \eta \|_{C^0}   \int_0^T \left|  \iint_{[0, 1]^2} \varphi(x) \psi(y) [W_{\G_n} - W] (t, x, y) dx dy \right| dt \\ &=
            \| \eta \|_{C^0}  \| \varphi \|_{C^0}  \| \psi \|_{C^0}  \int_0^T
         \left| \iint_{[0, 1]^2} \frac{\varphi(x)}{ \| \varphi \|_{C^0}}\frac{\psi(y)}{ \| \psi \|_{C^0}} [W_{\G_n} - W] (t, x, y) dx dy \right| dt   \\&
        \leq 4   \| \eta \|_{C^0}  \| \varphi \|_{C^0}  \| \psi \|_{C^0}  \int_0^T
          \| [W_{\G_n} - W] (t, \cdot, \cdot)  \|_{\Box} dt\,,
\end{split}
\end{equation*}
where the last step is a consequence of the properties \eqref{cut2} of the cut norm recalled in Appendix \ref{sec:graphons}. By~\eqref{e:conv_tloc_cut}, this yields~\eqref{limitemis} and concludes the proof of the lemma. 
\end{proof}
\begin{lemma} \label{l:stimabound}
{\sc Case 2.} Assume~\eqref{e:conv_tloc_cut} and~\eqref{bound_Wn}. Then the limit graphon $W$ satisfies~\eqref{bound_W} with $K_d=K_1$.
\end{lemma}
\begin{proof}
For a.e. $t \in [0, T],$ we have $W(t, \cdot, \cdot) \in L^1([0, 1]^2)$ and, owing to~\eqref{e:conv_tloc_cut} and up to subsequences, 
\begin{equation} \label{e:pointwise}
      \lim_{n \to + \infty} \left\| W_{\G_n}(t, \cdot, \cdot)-W(t, \cdot, \cdot)\right\|_{\Box} =0.
\end{equation}
For any such $t$, we fix $\hat x \in\, ]0, 1[$ Lebesgue point for the map $x \mapsto \displaystyle{\int_0^1 W(t, x, y) dy}$, which belongs to $L^1(0, 1)$. We fix $\ee < \min \{ \hat x, 1- \hat x \}$ and point out that, by combining~\eqref{e:pointwise} with the definition of cut norm~\eqref{cut1} we have 
\begin{equation} \label{e:Kee}
          \int_{\hat x-\ee}^{\hat x+ \ee} \int_0^1 W (t, x, y) dy dx =
          \lim_{n \to + \infty} \int_{\hat x-\ee}^{\hat x+ \ee} \int_0^1 W_{\G_n}(t, x, y) dy dx  
         \stackrel{\eqref{bound_Wn}}{\leq} 2 K_1 \ee, 
\end{equation}
which in turn implies owing to Lebesgue's Differentation Theorem 
$$
     \int_0^1 W (t, \hat x, y) dy  = \lim_{\ee \to 0^+}
     \frac{1}{2 \ee}\int_{\hat x-\ee}^{\hat x+ \ee} \int_0^1 W (t, x, y) dy dx 
    \stackrel{\eqref{e:Kee}}{\leq} K_1 
$$
and this yields~\eqref{bound_W} with $K_d=K_1$. 
\end{proof}
\subsection{A classical approximation result}
\label{subsec:24}
We now recall a  standard approximation result that will be often used in the next sections. For the sake of completeness,
 we  briefly sketch  its proof.
\begin{lemma} \label{l:average}
Assume $u \in L^1(0, 1)$, $U \in L^1 ([0, 1]^2)$, recall~\eqref{e:ij} and~\eqref{e:average}, set 
$$
    u^n_j = < u >_{I^n_j}, \qquad A^n_{jk} = < U >_{I^n_j \times I^n_k}
$$
and use the notation~\eqref{G-->W_G} and~\eqref{uG->u01}. Then
\begin{equation}
\label{e:convl1}
      u^n \to u \; \text{strongly in $L^1(0, 1)$}, \quad 
      W_{\G_n} \to U \; \text{strongly in $L^1([0, 1]^2)$}.
\end{equation}
If furthermore $u \in L^2(0, 1)$, $U \in L^2 ([0, 1]^2)$, then
\begin{equation}
\label{e:convl2}
      u^n \to u \; \text{strongly in $L^2(0, 1)$}, \quad 
      W_{\G_n} \to U \; \text{strongly in $L^2([0, 1]^2)$}.
\end{equation}
\end{lemma}
\begin{proof}
We only provide the proof of~\eqref{e:convl2}, the proof of \eqref{e:convl1}  being entirely analogous and slightly easier. \\
{\sc Step 1:} we first establish~\eqref{e:convl2} for $u \in C^0([0, 1])$. For fixed $\ee>0$,  by the uniform continuity of $u$ there is $n_\ee$ such that, for every $n \ge n_\ee$, if $|x-y| \leq n^{-1}$ then $|u(x) - u(y)| \leq \ee$. This implies that, for every $j=1, \dots, n$ and $x \in I_j^n$, we have 
$$
    |u(x) - u^n_j | = \left| u(x) - n \int_{I^n_j} u(y) dy \right| =  \left| n \int_{I^n_j}[u(x) - u(y)] dy \right|
    \leq n \int_{I^n_j}|u(x) - u(y)| dy \stackrel{n \ge n_\ee}{\leq} n \int_{I^n_j} \ee dy = \ee. 
$$
This in turn yields $\| u - u_n \|_{2 } \leq \ee$ and by the arbitrariness of $\ee$ we get the desired convergence result. \\
{\sc Step 2:} we consider the general case. For any $u, v \in L^2(0, 1)$ we term $u^n$ and $v^n$ the corresponding piecewise constant approximation, so that by using Jensen's inequality we get 
\begin{equation} \label{e:ennenne}
\begin{split}
  \| u^n - v^n \|^2_{2}& = \sum_{j=1}^n \int_{I^n_j} |u^n_j - v^n_j|^2 dx =
  \sum_{j=1}^n  \int_{I^n_j} \left( n \int_{I^n_j} [u-v] (y) dy \right)^2 dx \\& \stackrel{\text{Jensen}}{\leq} 
  \sum_{j=1}^n  \int_{I^n_j} n \int_{I^n_j} [u-v]^2(y)  dy dx  = \sum_{j=1}^n  \int_{I^n_j} [u-v]^2 (y) dy= \| u - v \|^2_{2}.
\end{split}
\end{equation}
We now fix $\ee>0$ and choose $v \in C^0([0, 1])$ in such a way that $\| u - v \|_{2} \leq \ee$. We then have  
\begin{equation*} 
    \| u - u^n \|_{2} \leq  \| u - v \|_{2} +   \| v - v^n \|_{2} +  \| v^n - u^n \|_{2}
     \stackrel{\eqref{e:ennenne}}{\leq} 2 \| u - v \|_{2} +   \| v- v^n \|_{2}  \leq 2\ee +   \| v- v^n \|_{2}.  
\end{equation*}
By using {\sc Step 1}, we can find $n_\ee$ such that if $n \ge n_\ee$ then $ \| v- v^n \|_{2}\leq \ee$ and by plugging this inequality into the above expression and using the arbitrariness of $\ee$ we obtain the desired convergence result. 
\end{proof}

\section{Convergence results for SIR model on graphs} 
\label{sec:cutconv}
This section contains the proofs of Proposition~\ref{p:uniex} and Theorem~\ref{THM 1}. The exposition is organized as follows:
\begin{itemize}
\item we first  establish in \S\ref{ss:puni}  the uniqueness part of Proposition~\ref{p:uniex}, namely the uniqueness of the distributional solution to \eqref{SIR_W}, \eqref{e:idSIR_W};
\item in \S\ref{ss:pthm1} we prove Theorem~\ref{THM 1}. As a byproduct, the proof provides a constructive argument that yields the existence part of Proposition~\ref{p:uniex}, contingent on showing a sequence of graphs satisfying~\eqref{e:conv_tloc_cut} and either~\eqref{e:bdl2} or~\eqref{bound_Wn};  
\item in \S\ref{ss:pex}, we prove the existence part of Proposition~\ref{p:uniex} by exhibiting a family of graphs satisfying the above requirements. This concludes the proof  Proposition~\ref{p:uniex}.
\end{itemize}
\subsection{Proof of the uniqueness part of Proposition~\ref{p:uniex}}\label{ss:puni}
We fix two distributional solutions $(s_1, i_1, r_1)$ and $(s_2, i_2, r_2)$ satisfying~\eqref{e:sirbound} and we separately consider {\sc Cases 1} and {\sc 2} below.

\smallskip
{\sc Case 1:} $W\in L^2\left([0,T] \times[0,1]^2;\mathbb{R}_+\right)$. We set $u_1: = s_1 + i_1$, $u_2: = s_2 + i_2$ and we recall that, owing to Lemma~\ref{l:reg}, $(s_1, i_1, r_1)$ and $(s_2, i_2, r_2)$ have Sobolev regularity and the equalities in~\eqref{SIR_W} are satisfied pointwise almost everywhere.
By using H\"older's and Young's inequalities we get  
\begin{equation} \label{e:uni1}
\begin{split}
        \f{d}{dt}  & \int_0^1 [s_1- s_2]^2(t, x) dx = \,2 \int_0^1 [s_1- s_2] [\partial_t s_1 - \partial_t s_2](t, x) dx\\
        \\& \,\stackrel{\eqref{SIR_W}}{=}- \underbrace{2 \int_0^1 [s_1-s_2]^2 (t, x) \int_0^1 \beta (t, y)W(t, x, y) i_1(t, y) dydx}_{\ge 0}\\
         &\qquad- 2 \int_0^1 s_2(t,x) [s_1-s_2] (t, x) \int_0^1\beta(t,y) W(t, x, y) [i_1 - i_2 ](t, y) dy dx \\ & 
        \stackrel{\text{H\"older}}{\leq} 2 \| s_2 \|_{\infty} \| \beta \|_{\infty} \| W(t, \cdot, \cdot) \|_{L^2([0,1]^2)} \| s_1(t, \cdot)  - s_2(t, \cdot)  \|_{L^2(0,1)} \| i_1(t, \cdot) - i_2 (t, \cdot)\|_{L^2(0,1)} \\
      & \stackrel{\eqref{e:sirbound}, i_i=u_i-s_i,\text{ Young}}{\leq} 2  \| \beta \|_{\infty}  \| W(t, \cdot, \cdot) \|_{L^2([0,1]^2)}
       \left[\frac{3}{2} \| s_1(t, \cdot)  - s_2(t, \cdot)  \|^2_{L^2(0,1)}  + \f{1}{2}  \| u_1(t, \cdot)  - u_2(t, \cdot)  \|^2_{L^2(0,1)} 
       \right].     
\end{split}
\end{equation}
Next, we point out that 
\begin{equation} \label{e:deru}
    \partial_t [u_1 - u_2] = \gamma [i_1 - i_2] = \gamma [u_1 - u_2 ] - \gamma [s_1 - s_2]
\end{equation}
and by using again H\"older's and Young's inequalities this yields 
\begin{equation} \label{e:uni2}
\begin{split}
          \f{d}{dt}  \int_0^1 [u_1- u_2]^2(t, x) dx & = \int_0^1  \gamma [u_1 (t, x) - u_2 (t,x) ]^2 dx - \int_0^1  \gamma [s_1 (t, x) - s_2(t, x) ] 
          [u_1 (t, x) - u_2 (t, x) ] dx \\ & 
           \stackrel{\text{H\"older, Young}}{\leq}
          \| \gamma \|_{\infty}        \left[\frac{3}{2} \| u_1(t, \cdot)  - u_2(t, \cdot)  \|^2_{L^2(0,1)}  + \f{1}{2}  \| s_1(t, \cdot)  - s_2(t, \cdot)  \|^2_{L^2(0,1)} 
       \right].     
\end{split}
\end{equation}
Combining~\eqref{e:uni1} and~\eqref{e:uni2} we get 
\begin{equation*} \begin{split}
   \f{d}{dt} &  \Big[ \| u_1(t, \cdot)  - u_2(t, \cdot)  \|^2_{L^2(0,1)}  +   \| s_1(t, \cdot)  - s_2(t, \cdot)  \|^2_{L^2(0,1)}  \Big]
   \\ & \leq \frac{7}{2}  \Big[  \| \beta \|_{\infty}  \| W(t, \cdot, \cdot) \|_{L^2([0,1]^2)} +   \| \gamma \|_{\infty}  \Big] \Big[  \| s_1(t, \cdot)  - s_2(t, \cdot)  \|^2_{L^2(0,1)} +  
    \| u_1(t, \cdot)  - u_2(t, \cdot)  \|^2_{L^2(0,1)}  \Big]
    \end{split}
\end{equation*}
and by Gr\"onwall Lemma this implies $ \| u_1(t, \cdot)  - u_2(t, \cdot)  \|^2_{L^2(0,1)}  +   \| s_1(t, \cdot)  - s_2(t, \cdot)  \|^2_{L^2(0,1)}  = 0$ for every $t \in [0,T]$. Since $i_1= u_1-s_1$, $i_2= u_2-s_2$ this in turn implies $s_1 = s_2$, $i_1 = i_2$ a.e. on $[0,T] \times [0,1]$ and, since $r_1 = 1-s_1 - i_1$, $r_2 = 1-s_2 -i_2$, it also implies $r_1 = r_2$ a.e. on $[0,T] \times [0,1]$. 

\smallskip
{\sc Case 2:} $W\in L^1\left([0, T] \times[0,1]^2;\rr_+\right)$ satisfies~\eqref{bound_W}. As before we set $u_1: = s_1 + i_1$, $u_2: = s_2 + i_2$ and we get 
\begin{equation} \label{e:uni3}
\begin{split}
      \f{d}{dt}  |s_1  - s_2 | (t, x)  = &\,
      \mathrm{sign}[s_1-s_2] [\partial_t s_1 - \partial_t  s_2 | (t, x) \\ 
       = &\,- \underbrace{  \mathrm{sign}[s_1-s_2] [s_1-s_2](t, x)  \int_0^1 \beta (t, y) W(t, x, y) i_1(t, y) dy }_{\ge 0} \\ 
       &\quad +  \mathrm{sign}[s_1-s_2] s_2 (t, x) \int_0^1 \beta (t, y)W(t, x, y)  [i_1- i_2](t, y) dy\\
       \stackrel{\eqref{bound_W},\eqref{e:sirbound}, W\geq0}{\leq} &\, K_d  \| \beta \|_{\infty}  \| [i_1 - i_2](t,\cdot) \|_{L^\infty(0,1)} \\ 
       \stackrel{\eqref{bound_W},\eqref{e:sirbound}}{\leq} &\, K_d  \| \beta \|_{\infty}  \Big[ \| [s_1 - s_2](t,\cdot) \|_{L^\infty(0,1)}
       +\| [u_1 - u_2](t,\cdot) \|_{L^\infty(0,1)} \Big].
\end{split}
\end{equation}
By \eqref{e:deru} we get
\begin{equation}\label{e:uni4}
          \f{d}{dt}  \|[u_1  - u_2](t,\cdot) \|_{L^\infty(0,1)}\leq \| \gamma \|_{\infty}  \Big[ \| [s_1 - s_2](t,\cdot) \|_{L^\infty(0,1)}
       +\| [u_1 - u_2](t,\cdot) \|_{L^\infty(0,1)} \Big]
\end{equation}
and combining~\eqref{e:uni3} and~\eqref{e:uni4} with Gr\"onwall Lemma we get $\| [s_1 -s_2](t,\cdot) \|_{L^\infty(0,1)}=0$ and $ \| [u_1 -u_2](t,\cdot) \|_{L^\infty(0,1)}=0$ for every $t \in [0, T]$. Arguing as in the previous case this in turn implies $s_1 = s_2$, $i_1 = i_2$, $r_1 = r_2$ a.e. on $[0, T] \times [0,1]$.

\subsection{Proof of Theorem~\ref{THM 1}}  \label{ss:pthm1}
We now provide the proof of Theorem~\ref{THM 1}.
Let $\{\G_n \}_{n \in \nn}$, $\{\beta^n \}_{n \in \nn}$ and $\{\gamma^n \}_{n \in \nn}$ be fixed as in the statement of the theorem.  Recalling the notation~\eqref{uG->u01} we observe that $(s^n, i^n, r^n)$ is a distributional solution of 
\begin{equation}
	\label{SIR_Wn}
	\begin{cases}
	\partial_t s^n(t,x)=-s^n(t,x)
     \displaystyle{\int_0^1\beta^n(t,y) {W}_{\G_n}(t,x,y)i^n(t,y)\,dy} & \\
	\partial_t i^n(t,x)=s^n(t,x)\displaystyle{
       \int_0^1\beta^n(t,y) {W}_{\G_n}(t,x,y)i^n(t,y)\,dy-\gamma^n(t,x)i^n(t,x)} & x\in[0,1]\\
	\partial_t r^n(t,x)=\gamma^n(t,x)i^n(t,x) \phantom{\displaystyle{\int}}
	\end{cases}
\end{equation}
and satisfies the initial condition $
	        s^n(0, \cdot) = s_0^n$,  $i^n(0, \cdot) = i_0^n$, $r^n(0, \cdot) = r_0^n$.
This implies that for every $(\varphi,\psi,\eta)\in \left(C_c^\infty\left([0, T]\times[0,1]\right)\right)^3$ we have 
	\begin{equation}\begin{split}
	\label{eq:weakLinfty2}	
    & \int_0^T \! \! \! \int_0^1\! \! \!   s^n(t,x)\left(\partial_t\varphi(t,x)-\varphi(t,x)\int_0^1\beta^n(t,y){W}_{\G_n}(t,x,y)i^n(t,y)\,dy\right)dxdt=
    \int_0^1s^n_0(x)\varphi(0,x)\,dx \\ 	      
    & \int_0^T \! \! \! \!  \int_0^1\! \! \! i^n(t,x)\left(\partial_t\psi(t,x)-\gamma^n(t,x)\psi(t,x)\right)
    \!  + \! 
    s^n(t,x)\psi(t,x)\int_0^1\! \! \! \beta^n(t,y){W}_{\G_n}(t,x,y)i^n(t,y)\,dydxdt \\
    & \qquad \qquad \qquad  \qquad \qquad \qquad \qquad \qquad \qquad  \qquad \qquad \qquad= \!\int_0^1
    \! \! \! i^n_0(x)\psi(0,x)\,dx  \\
	&
     \int_0^T \int_0^1r^n(t,x)\partial_t\eta(t,x)+\gamma^n(t,x)\eta(t,x)i^n(t,x)dxdt=
     \int_0^1r^n_0(x)\eta(0,x)\,dx\,.  \\
     \end{split}
	\end{equation}
We recall the bounds~\eqref{e:01SIR_G}, which owing to the notation~\eqref{uG->u01} imply 
\begin{equation}\label{e:mp}
\| s^n \|_{\infty},\| i^n \|_{\infty}, \| r^n \|_{\infty}\leq 1
\end{equation}
and we conclude  that, up to subsequences, 
\begin{equation}
\label{e:weaks}
     s^n \rightharpoonup^\ast s, \quad  i^n \rightharpoonup^\ast i, \quad  r^n \rightharpoonup^\ast r \quad \text{weakly$^\ast$ in $L^\infty([0, T] \times [0, 1])$},
\end{equation}
for some limit function $(s, i, r) \in L^\infty([0, T] \times [0, 1]; \rr^3).$	Also, 
\begin{equation} \label{e:convidat}
     s_0^n \rightharpoonup^\ast s_0, \quad  i_0^n \rightharpoonup^\ast i_0, \quad  r_0^n \rightharpoonup^\ast r_0 \quad \text{weakly$^\ast$ in $L^\infty(0, 1)$},
\end{equation}
for some limit function $(s_0, i_0, r_0) \in L^\infty( [0, 1]; \rr^3).$	 We now pass to the limit in~\eqref{eq:weakLinfty2} by separately considering the two cases. 

\smallskip
{\sc Case 1:} we assume~\eqref{e:bdl2} and proceed according to the following steps.

{\sc Step 1A:} since $(s^n, i^n, r^n)$ is a distributional solution of~\eqref{SIR_Wn}, by \eqref{e:boundder} and recalling~\eqref{e:bdl2} and~\eqref{e:hpgammbeta} we conclude that 
\begin{equation}
\label{e:bounder3}
      \| \partial_t s^n \|_{2},    \| \partial_t i^n \|_{2},    \| \partial_t r^n \|_{2} \leq 
      C( K_0, M).
\end{equation}
We now show that this implies that, up to subsequences, for every $ \psi \in L^2 ([0, T] \times [0, 1]^2)$
\begin{equation}
\label{e:weakpr}
       \int_0^T \! \! \int_0^1\! \! \int_0^1 \! \! \psi (t, x, y)  s^n (t, x) i^n(t, y) dx dy dt \! \!\, \to \! \!  \int_0^T\! \! \int_0^1\! \!  \int_0^1  \! \!
       \psi (t, x, y)  s (t, x) i(t, y) dx dy dt, 
\end{equation}
 that is the product $s^n(t,x) i^n(t,y)$ weakly converges to $s(t,x) i(t,y)$ in $ L^2 ([0, T] \times [0, 1]^2)$.  To this end, we point out that for every fixed $n \in \nn$, since $\partial_t s^n, \partial_t i^n, \partial_t r^n$ are all $L^2$ functions owing to~\eqref{e:bounder3}, we can select a representative of $(s^n, i^n, r^n)$ (which a priori as $L^\infty$ functions are only defined up to negligible sets) such that the map $t \mapsto (s^n, i^n, r^n)(t, \cdot)$ is continuous from $[0, T]$ to $L^2 (0, 1)$ endowed with the strong topology. In the following we will always work with this representative, which allows us to give a meaning to the value $(s^n, i^n, r^n)(t, \cdot) \in L^2 (0, 1)$ for \emph{every} $t \in [0, T]$.

Using the $L^\infty$ bounds  and a standard diagonal argument we can extract a further subsequence such that 
\begin{equation}
\label{e:q}
     s^n (q, \cdot) \rightharpoonup^\ast s(q, \cdot), \quad  i^n(q, \cdot) \rightharpoonup^\ast i(q, \cdot), \quad  
     r^n (q, \cdot) \rightharpoonup^\ast r (q, \cdot) \; \text{weakly$^\ast$ in $L^\infty( 0, 1)$, for every $q \in [0, T] \cap \mathbb{Q}$. }
\end{equation}
Next, we fix $t \in [0, T]$ and a test function $\psi \in L^2 ([0, T] \times [0, 1]^2)$. We also fix a sequence of rational numbers $q_k \to t$ as $k \to + \infty$. We then have 
\begin{equation}\label{e:irrt}
\begin{split}
     & \left| \iint_{[0, 1]^2} \psi (t, x, y)  s^n (t, x) i^n(t, y) dx dy -  \iint_{[0, 1]^2} \psi (t, x, y)  s (t, x) i(t, y) dx dy      
      \right|\\&
     \leq \underbrace{\left| \iint_{[0, 1]^2} \psi (t, x, y)  s^n (t, x) i^n(t, y) dx dy-
     \iint_{[0, 1]^2} \psi (t, x, y)  s^n (q_k, x) i^n(t, y) dx dy \right|}_{E^n_k} \\ & +
     \underbrace{ \left| \iint_{[0, 1]^2} \psi (t, x, y)  s^n (q_k, x) i^n(t, y) dx dy-
     \iint_{[0, 1]^2} \psi (t, x, y)  s^n (q_k, x) i^n(q_k, y) dx dy \right|}_{B^n_k} \\ & +
   \underbrace{  \left| \iint_{[0, 1]^2} \psi (t, x, y)  s^n (q_k, x) i^n(q_k, y) dx dy - \iint_{[0, 1]^2} \psi (t, x, y)  s (q_k, x) i(q_k, y) dx dy
       \right| }_{C^n_k}\\ & +
     \underbrace{ \left| \iint_{[0, 1]^2} \psi (t, x, y)  s (q_k, x) i(q_k, y) dx dy - \iint_{[0, 1]^2} \psi (t, x, y)  s (t, x) i(t, y) dx dy \right|}_{D^n_k}.
\end{split}
\end{equation}
We now estimate each of the above terms separately. Note that 
\begin{equation}
\label{e:A}
   E^n_k \leq \| \psi(t, \cdot, \cdot) \|_{L^2 ([0, 1]^2)} \underbrace{\| i^n \|_{\infty}}_{\leq 1} \| s^n(t, \cdot) - s^n (q_k, \cdot) \|_{L^2 (0, 1)} 
\end{equation}
and, owing to Jensen's inequality,  
\begin{equation} \label{e:jensen2}
\begin{split}
    & \| s^n(t, \cdot) - s^n (q_k, \cdot) \|^2_{L^2 (0, 1)}  = \int_0^1 [s^n(t, \cdot) - s^n (q_k, \cdot) ]^2 dx =
    \int_0^1 \left[\int_{q_k}^t \partial_\tau s^n(\tau, \cdot) d\tau \right]^2 dx \\
    & \stackrel{\text{Jensen}}{\leq}
    |q_k - t| \int_0^1 \int_{q_k}^t (\partial_\tau s^n(\tau, \cdot) )^2 d\tau dx \stackrel{\eqref{e:bounder3}}{\leq} C(K_0, M) |q_k - t|\,,
\end{split}
\end{equation} 
so that plugging the above inequality into~\eqref{e:A} we get 
$$
E_n^k \leq    C(K_0, M)  \| \psi(t, \cdot, \cdot) \|_{L^2 ([0, 1]^2)}\sqrt{ |q_k - t|}, \qquad \mbox{ for every $n$.}
$$
By an analogous argument, we get   $B_n^k \leq    C(K_0, M)  \| \psi(t, \cdot, \cdot) \|_{L^2 ([0, 1]^2)} \sqrt{|q_k - t|}$. 

Owing to~\eqref{e:bounder3}, we have 
\[
   \partial_t s^n \rightharpoonup \partial_t s, \; \partial_t i^n \rightharpoonup \partial_t i \quad \text{weakly in $L^2([0, T] \times [0, 1]),$}
\]
and by recalling~\eqref{e:jensen2} and using the lower semicontinuity of the norm with respect to the weak convergence we get 
\[
    \int_0^1 \int_{q_k}^t (\partial_\tau s(\tau, \cdot) )^2 d\tau dx \leq
      C(K_0, M), \qquad  \int_0^1 \int_{q_k}^t (\partial_\tau i(\tau, \cdot) )^2 d\tau dx \leq
     C(K_0, M).
\]
Hence, arguing as before this yields $D^n_k \leq   C(K_0, M)  \| \psi(t, \cdot, \cdot) \|_{L^2 ([0, 1]^2)} \sqrt{|q_k - t|}$, for every $n \in \nn$. 
To conclude, we fix $\ee>0$ and choose $k$ in such a way that $E^n_k + B^n_k + D^n_k \leq \ee$, for every $n$. Next, we recall~\eqref{e:q} and choose $n_\ee$ in such a way that $C^n_k \leq \ee$ for every $n \ge n_\ee$. By the arbitrariness of $\ee$, this implies that the left hand side of~\eqref{e:irrt} vanishes in the $n \to + \infty$ limit for a.e. $t \in [0, T]$. By Lebesgue's Dominated Convergence Theorem this yields~\eqref{e:weakpr}.  \\
{\sc Step 1B:} we show that the limit $(s, i, r)$ in~\eqref{e:weaks}  is a distributional solution of~\eqref{SIR_W}, \eqref{e:idSIR_W}. We subtract~\eqref{eq:weakLinfty} from~\eqref{eq:weakLinfty2} and in particular we get
\begin{equation}\label{e:diffdf}
\begin{split}
\left|\int_0^T\right.&\left.\int_0^1\left(s^n(t,x)\int_0^1\beta^n(t,y){W}_{\G_n}(t,x,y)i^n(t,y)\,dy-s(t,x)\int_0^1\beta(t,y)W(t,x,y)i(t,y)\,dy\right)\varphi(t,x)\,dxdt\right| \\
\leq&\underbrace{\left|\int_0^T\int_0^1\int_0^1\left({W}_{\G_n}(t,x,y)-W(t,x,y)\right)\beta^n(t,y)s^n(t,x)i^n(t,y)\varphi(t,x)\,dydxdt\right|}_{I_1}\\
+& \underbrace{\left|\int_0^T\int_0^1\int_0^1W(t,x,y)\left(\beta^n(t,y)s^n(t,x)i^n(t,y)-
\beta(t,y)s(t,x)i(t,y)\right)\varphi(t,x)\,dydxdt\right|}_{I_2}
\end{split}
\end{equation}
for every $\varphi \in C^\infty_c ([0, T] \times [0, 1]^2)$.  We first control $I_2$:
\begin{equation*}
\begin{split}
 I_2  & \leq \underbrace{\left|\int_0^T\int_0^1\int_0^1 W(t,x,y) [ \beta^n(t,y) - \beta (t, y)] s^n(t,x)i^n(t,y)\varphi(t,x)\,dydxdt\right|}_{: = I_{21}} \\
 &+    \underbrace{ \left|\int_0^T\int_0^1\int_0^1 W(t,x,y)  \beta (t, y) [s^n(t,x)i^n(t,y) - s (t, x) i(t, y)] \varphi(t,x)\,dydxdt\right|}_{: = I_{22}}. 
\end{split}
\end{equation*}
To see that $I_{22}$ vanishes in the $n \to + \infty$ limit it suffices to point out that $W \beta \in L^2 ([0, T] \times [0, 1]^2)$ and then use~\eqref{e:weakpr}. To deal with $I_{21}$ we 
observe that 
\begin{equation*}
\begin{split}
     I_{21} &  \stackrel{\| s_n \|_{\infty}, \| i_n \|_{\infty} \leq 1}{\leq} \int_0^T \int_0^1 \int_0^1 W(t, x, y) |\beta^n(t,y) - \beta (t, y)| dy dx dt \\
     & \stackrel{\text{H\"older}}{\leq} \| W \|_{L^2 ([0, T] \times [0, 1]^2)} \left( \int_0^T   \int_0^1 |\beta^n(t,y) - \beta (t, y)|^2 dy dt    \right)^{1/2} \\
     &  \stackrel{\eqref{e:bdl2}}{\leq} K_0 \left( \int_0^T   \int_0^1 |\beta^n(t,y) - \beta (t, y)|^2 dydt    \right)^{1/2} \stackrel{\eqref{e:hpgammbeta}}{\leq}
       \sqrt{2 M} K_0  \left( \int_0^T  \int_0^1  |\beta^n(t,y) - \beta (t, y) | dy dt    \right)^{1/2}
\end{split}
\end{equation*}    
and by~\eqref{e:hpgammbeta} the right-hand side of the above equation vanishes in the $n \to + \infty$ limit. Going back to~\eqref{e:diffdf}, this shows that $I_{2}$ converges to $0$. 
Let us focus now on $I_1$. We set
$$
    L: = \|\beta^n\|_{\infty}\|s^n\|_{\infty}\|i^n\|_{\infty}\|\varphi\|_{\infty} \stackrel{\|s^n\|_{\infty},\|i^n\|_{\infty}\leq 1}{\leq} \|\beta\|_{\infty} \|\varphi\|_{\infty} \stackrel{\eqref{e:betagamma}}{\leq} M \|\varphi\|_{\infty}  \, .  
$$ 
By multiplying and dividing $I_1$ by $L$ and making use of \eqref{cut2} in Appendix \ref{sec:graphons}, we get 
\[
\begin{split}
I_1=&\,L\left|\int_0^T\int_0^1\int_0^1\right.\left({W}_{\G_n}(t,x,y)-W(t,x,y)\right) \cdot\left.  \f{s^n(t,x)}{\|s^n\|_{\infty}}\f{\varphi(t,x)}{\|\varphi\|_{\infty}} \f{\beta^n(t,y)}{\|\beta^n\|_{\infty}}\f{i^n(t,y)}{\|i^n\|_{\infty}}\,dydxdt\right|\\
\leq&\,L\int_0^T\sup_{\substack{\|f(t,\cdot)\|_{\infty}\leq1 \\ \|g(t,\cdot)\|_{\infty}\leq1}}\left|\int_{[0,1]^2}({W}_{\G_n}(t,x,y)-W(t,x,y))f(t,x)g(t,y)\,dydx\right|\,dt\\
\stackrel{\eqref{cut2}}{\leq}&4L \int_0^{T}\|{W}_{\G_n}-W\|_{\Box}\,dt \stackrel{\eqref{e:conv_tloc_cut}}{\to} 0 \quad \text{as $n \to + \infty$.}
\end{split}
\]
This implies that $I_1\to0$ as $n\to+\infty$ and hence that the left hand side of~\eqref{e:diffdf} vanishes in the $n \to + \infty$ limit.  By relying on similar arguments, one can pass to the limit in all the other terms in~\eqref{eq:weakLinfty2} and show that $(s, i, r)$ is a distributional solution of~\eqref{SIR_W},~\eqref{e:idSIR_W}. 

\smallskip
{\sc Case 2:} we assume the family of graphs satisfy~\eqref{bound_Wn}.  The proof follows the same argument as in the previous case, so we only provide a sketch and highlight the points where there is some difference.  By using the proof of Lemma~\ref{l:reg} and recalling~\eqref{e:hpgammbeta} we conclude that 
$$
      \| \partial_t s^n \|_{\infty},    \| \partial_t i^n \|_{\infty},    \| \partial_t r^n \|_{\infty} \leq 
      C( K_1, M),
$$
which implies that there is a representative of $(s, i, r)$ such that the maps $t \mapsto s(t, \cdot)$, $t \mapsto i(t, \cdot)$, $t \mapsto r(t, \cdot)$  are continuous from $[0, T]$ in $L^\infty (0, 1)$ endowed with the $\esssup$ (strong) topology. The goal is now to show that this implies that, up to subsequences, for every $\psi \in L^1 ([0, T] \times [0, 1]^2)$
\begin{equation}
\label{e:weakpr2}
       \int_0^T \! \! \int_0^1\! \! \int_0^1 \! \! \psi (t, x, y)  s^n (t, x) i^n(t, y) dx dy dt \! \! \,\to \! \!  \int_0^T\! \! \int_0^1\! \!  \int_0^1  \! \!
       \psi (t, x, y)  s (t, x) i(t, y) dx dy dt.
\end{equation}
The proof follows the same lines as the proof of~\eqref{e:weakpr}: up to subsequences, we have~\eqref{e:q} and hence, for a given $t \in [0, T]$, we have the splitting as in~\eqref{e:irrt}, where now $$
 E^n_k \leq \| \psi(t, \cdot, \cdot) \|_{L^1([0, 1]^2)} \| s^n(t, \cdot) - s^n (q_k, \cdot) \|_{L^\infty (0, 1)}
 $$
 and by using the continuity of the map    $t \mapsto s(t, \cdot)$ from $[0, T] \to L^\infty (0, 1)$ we conclude that $\| s^n(t, \cdot) - s^n (q_k, \cdot) \|_{L^\infty (0, 1)}$ is arbitrarily small, provided $q_k$ is sufficiently close to $t$. By reasoning similarly we can show that the other terms on the right hand side of~\eqref{e:irrt} are arbitrarily small and hence establish~\eqref{e:weakpr2}. The rest of the proof in {\sc Case 2} works as in {\sc Case 1}. 

\begin{remark}
	\label{rem:si_Linfty}
     In system~\eqref{SIR_G} the term $s_j^n$  only interacts with $n^{-1} \sum_{k=1}^n \beta_k^n(t)A_{jk}^n i_k^n$. This allows us to pass to the limit in the integral term in the first equation of \eqref{SIR_W}. From the point of view of applications it would be also meaningful to consider the following SIR system:
	\begin{equation}
		\label{SIR'_G}
		\begin{cases}
\displaystyle{\f{ds_j(t)}{dt}=-s_j(t)\f1{n}\sum_{k=1}^{n}\beta_k^n(t)
A_{jk}^n(t)i_k(t)-\beta_j^n(t)s_j^n(t)i_j^n(t)} & \\
		\displaystyle{\f{di_j(t)}{dt}=s_j(t)\f1{n}\sum_{k=1}^{n}\beta_k^n(t)A_{jk}^n(t)i_k(t)+\beta_j^n(t)s_j^n(t)i_j^n(t)-\gamma_j^n(t) i_j(t)} &\qquad\qquad\qquad j=1,\,\dots,\,n\\
		\displaystyle{\f{dr_j(t)}{dt}=\gamma_j^n(t) i_j(t)} 
		\end{cases}
	\end{equation}
	where the interaction of the node $j$ with itself is not scaled by the factor $n^{-1}$ since it is somehow independent from the size of the network.
The proof of Theorem~\ref{THM 1} does not directly extend to~\eqref{SIR'_G} because from the weak$^*$ convergence of $s^n$ and $i^n$ alone we cannot infer that the product $s^n(t,x)i^n(t,x)$ converges to $s(t,x)i(t,x)$. As a matter of fact, the argument in the proof of Theorem~\ref{THM 1} shows that there is $\pi\in L^\infty\left([0,T]\times [0,1]\right)$, $0 \leq \pi \leq 1$, such that $(s^n,i^n, r^n)$ converge weakly$^*$ in $L^\infty\left([0,T] \times [0,1]\right)$ to a distributional solution of 
	\[
	\begin{cases}
	\displaystyle{\partial_t s(t,x)=-s(t,x)\int_0^1\beta(t,y)W(t,x,y)i(t,y)\,dy-\beta(t,x)\pi(t,x)}& \\
	\displaystyle{\partial_t i(t,x)=s(t,x)\int_0^1\beta(t,y)W(t,x,y)i(t,y)\,dy+\beta(t,x)\pi(t,x)-
\gamma(t,x)i(t,x)} & \qquad\qquad x\in[0,1]\\
	\partial_tr(t,x)=\gamma(t,x)i(t,x) \phantom{\displaystyle{\int}}
     \phantom{\displaystyle{\int}}
	\end{cases}
	\]
     satisfying the initial condition~\eqref{e:idSIR_G}.
	\end{remark}
	\subsection{Existence of distributional solution of \eqref{SIR_W}}\label{ss:pex}
We can now provide the proof of the existence part of Proposition~\ref{p:uniex}.
Let $W$, $T$, $\gamma$, $\beta$, $s_0, i_0, r_0$ be as in the statement of Proposition~\ref{p:uniex}.
Owing to the  proof of Theorem~\ref{THM 1}, to establish the existence part of Proposition~\ref{p:uniex} it suffices to exhibit sequences of graphs
$\{ \G_n \}_{n \in \mathbb N}$, coefficients $\{ \beta^n \}_{n \in \mathbb N}$ and $\{ \gamma^n \}_{n \in \mathbb N}$ and initial data $\{ (s_0^n, i_0^n, r_0^n)\}_{n \in \mathbb N}$ such that~\eqref{e:conv_tloc_cut},~\eqref{e:hpgammbeta},~\eqref{e:convidat} and either~\eqref{e:bdl2} or~\eqref{bound_Wn} are satisfied. 

We define $\beta^n$, $\gamma^n$ and $(s_0^n, i_0^n, r_0^n)$ by using~\eqref{e:average}, \eqref{e:betagamma} and the notation~\eqref{uG->u01}. Note that, if $\beta^n_k$ and $\gamma^n_k$ are as in~\eqref{e:betagamma}, then
\begin{equation}
\label{linftybk}
    |\beta^n_k(t)|\leq  \| \beta \|_{\infty}, \quad |\gamma^n_k (t)| \leq \| \gamma \|_{\infty}, \quad \text{for every $k=1, \dots, n$ and a.e. $t \in [0, T]$} 
\end{equation}
and, by using~\eqref{e:convl1},
$$
        \lim_{n \to +\infty}  \| \gamma(t, \cdot)  - \gamma^n 
         (t, \cdot) \|_{L^1(0, 1)} =0, \quad
        \lim_{n \to +\infty}  \| \beta(t, \cdot)  - \beta^n 
         (t, \cdot) \|_{L^1(0, 1)} =0 \qquad \text{for a.e. $t \in [0, T]$.}
$$
Owing to Lebesgue's Dominated Convergence Theorem and to~\eqref{linftybk} this yields 
\begin{equation}
\label{e:settembre}
           \lim_{n \to +\infty} \int_0^T \| \beta(s, \cdot)  - \beta^n 
         (s, \cdot) \|_{L^1(0, 1)} ds =0,
           \quad 
          \lim_{n \to +\infty} \int_0^T \| \gamma(s, \cdot)  - \gamma^n 
         (s, \cdot) \|_{L^1(0, 1)} ds =0.
\end{equation}
This implies that both conditions in~\eqref{e:hpgammbeta} are satisfied. 
To establish~\eqref{e:convidat} we use again~\eqref{e:convl1}. 
We now verify the assumptions on $W_{\G_n}$. If $W \in L^2 ([0, T] \times [0, 1]^2;\mathbb{R}_+)$, then by Jensen's inequality
\begin{equation}\label{boundl21}
\begin{split}
   & \| W_{\G_n}(t, \cdot,\cdot) \|^2_{L^2 ([0, 1]^2)} 
    = \sum_{j, k=1}^n \iint_{I^n_j \times I^n_k} (A^n_{jk})^2 dx dy \stackrel{\eqref{e:adj}}{=}
   \sum_{j, k=1}^n \iint_{I^n_j \times I^n_k} ( < W(t, \cdot, \cdot) >_{I^n_j \times I^n_k})^2 dx dy\\
    &  \stackrel{\text{Jensen}}{\leq}
     \sum_{j, k=1}^n \iint_{I^n_j \times I^n_k}  < W^2 (t, \cdot, \cdot) >_{I^n_j \times I^n_k} dx dy = 
     \sum_{j, k=1}^n  \iint_{I^n_j \times I^n_k} W^2 (t, x, y)dx dy  = \| W (t, \cdot, \cdot) \|^2_{L^2 ([0, 1]^2)}
\end{split}
\end{equation} 
for a.e. $t \in [0, T]$, and this yields~\eqref{e:bdl2} with $K_0 = \| W \|_{2}$. If $W\in L^1([0,T]\times [0,1]^2;\rr_+)$ satisfies~\eqref{bound_W}, we have 
    \begin{equation}
    \label{boundl11}
   \begin{split}
   \esssup_{t \in [0, T], \, x \in [0,1]}& \int_0^1W_{\G_n} (t,x,y)\,dy
       \stackrel{\eqref{e:adj}}{=} 
       \esssup_{t \in [0, T], \, j=1, \dots, n } \sum_{k=1}^n \int_{I^n_k}  
       < W(t, \cdot,\cdot) >_{I^n_j \times I^n_k}\,dy \\ & 
     =
     \esssup_{t \in [0, T], \, j=1, \dots, n } \sum_{k=1}^n n
  \iint_{I^n_j \times I^n_k} W (t, x, y)dx dy  =
     \esssup_{t \in [0, T], \, j=1, \dots, n } n
  \int_{I^n_j } \int_0^1 W (t, x, y)dy dx \\
&  \stackrel{\eqref{bound_W}}{\leq}
      \esssup_{t \in [0, T], \, j=1, \dots, n } n
  \int_{I^n_j } K_d \ dx = K_d, 
\end{split}
\end{equation}
which yields~\eqref{bound_Wn} with $K_1=K_d$. We are left to verify~\eqref{e:conv_tloc_cut}. To this end, we recall that, owing to~\eqref{cut2} in Appendix \ref{sec:graphons}, 
\begin{equation}
\label{e:cutl1}
         \| U(t, \cdot,\cdot) \|_{\Box} \leq \| U(t, \cdot,\cdot) \|_{L^1 ([0, 1]^2)}
         \quad \text{for a.e. $t \in [0, T]$}
\end{equation}
for every time-dependent graphon $U \in L^1 ([0 , T] \times [0, 1]^2)$. Next, we combine~\eqref{e:convl1} and either \eqref{boundl21} or \eqref{boundl11} with Lebesgue's Dominated Convergence Theorem to get 
 $$
    \lim_{n \to +\infty} \int_0^T \| W_{\G_n}(s, \cdot,\cdot)  - W
         (s, \cdot,\cdot) \|_{L^1([0, 1]^2)} ds =0,
$$
which owing to~\eqref{e:cutl1} yields~\eqref{e:conv_tloc_cut}. 

\section{Deterministic Approximation to SIR on Graphons}\label{s:det}
This section presents the convergence analysis for the deterministic approximations to \eqref{SIR_W}, i.e. the proof of Theorem~\ref{thm:det}. We consider each case of the theorem separately.
\subsection{Proof of Theorem~\ref{thm:det} ({\sc Case 1}) } 
\label{sec:proofdet}
We first prove the following a priori estimate on the $L^{2}$-error of the approximation at each fixed time.
\begin{lemma}\label{l:gronwall}
Under the same assumptions as in the statement of Theorem~\ref{thm:det} ({\sc Case 1}) we have
\begin{equation}
\label{e:grolem}
\begin{split}
       \|s (t, \cdot) -  s^n (t, \cdot)  \|^2_{L^2(0,1)} & 
     \! \! +  \|i (t, \cdot) -  i^n (t, \cdot)  \|^2_{L^2(0,1)} + 
     \|r (t, \cdot) -  r^n (t, \cdot)  \|^2_{L^2(0,1)}
      \\
     \leq&\, 
     B \Big(
       \|s_0 - s^n_0  \|^2_{2} \! \!
     +  \|i_0 -  i^n_0  \|^2_{2} \! \!+
     \|r_0 -  r^n_0  \|^2_{2} \Big)
      \\ & + 
      B \displaystyle{\int_0^t
       \Big( \| [\beta W -
      \beta^nW_{\G_n} ] (s,\cdot, \cdot) \|^2_{L^2([0,1]^2)}
     +  \| \gamma (s, \cdot)- \gamma^n (s, \cdot) \|^2_{L^2(0,1)}
     \Big) ds}
\end{split}
\end{equation}
for every $t \in [0, T]$, where $B$ is a suitable constant only depending on $T,  \| \beta \|_{\infty},
 \| \gamma \|_{\infty}$ and $\| W \|_{2}$. 
\end{lemma}

\begin{proof}
We rely on a Gr\"onwall type argument and  proceed according to the following steps. 

\smallskip
{\sc Step 1:} we subtract the first line of~\eqref{SIR_Wn} from the first line of~\eqref{SIR_W} and multiply the result by~$2 [s - s^n] $. This yields 
\begin{equation}\label{e:smenos}
\begin{split}
       \partial_t [s -  s^n]^2 (t, x)  = &
        \underbrace{- 2 [s -  s^n]^2
       (t, x)  \int_0^1 \beta (t, y) W (t,x,y) i (t, y) dy}_{J_1(t, x)}  \\
       & \underbrace{- 2
       s^n(t,x) [s -  s^n](t, x)
       \int_0^1  \left( \beta (t, y) W (t,x,y) i (t, y)  -
       \beta^n (t, y) W_{\G_n} (t, x,y)  i^n (t, y)
      \right) dy}_{J_2(t, x)}. 
\end{split}
\end{equation}  
Recalling that $\beta, W \ge 0$ by assumption and that $i\ge 0$ owing to~\eqref{e:sirbound} we have 
\begin{equation}
\label{e:j1}
       J_1 (t, x) \leq 0 \quad \text{for a.e. $(t, x) \in [0, T] \times [0, 1]$.}
\end{equation}
To control $J_2$, we split it as
\begin{equation}\label{e:j2}
\begin{split}
       | J_2(t, x)| \leq &
      \underbrace{2  \left|
        s^n(t,x) [s -  s^n] (t, x)
       \int_0^1   \beta (t, y) W (t, x, y) [i -  i^n](t, y) dy\right|}_{J_{21}} \\
      & + 
      \underbrace{ \left| 2
       s^n(t,x) [s - s^n] (t, x)
       \int_0^1    i^n (t,y)[\beta (t, y) W (t, x,y)  -  \beta^n (t, y) W_{\G_n}  (t, x,y) ]dy \right|}_{J_{22}}\,.
\end{split}
\end{equation}
By using H\"older's and Young's  inequalities we get 
\begin{equation}\label{e:j21}
\begin{split}
          \int_0^1 J_{21} (t, x) dx  & \stackrel{\eqref{e:mp}}{\leq}
          2 \int_0^1 
         |s -  s^n| (t, x)
       \int_0^1   \beta (t, y) W (t, x, y) |i -  i^n| (t, y) dy dx \\ & 
        \stackrel{\text{H\"older}}{\leq}
      2 \| \beta \|_{\infty} \| [s-s^n] (t, \cdot)\|_{L^2(0, 1)} \| [i - i^n](t, \cdot) \|_{L^2(0, 1)}
      \| W (t, \cdot, \cdot) \|_{L^2 ([0, 1]^2)}
     \\ & 
        \stackrel{\text{Young}}{\leq}
       \| \beta \|_{\infty}
      \| W (t, \cdot, \cdot) \|_{L^2 ([0, 1]^2)} 
     \left[ \|[s-s^n ](t, \cdot)\|^2_{L^2(0, 1)}+ \| [ i - i^n](t, \cdot) \|^2_{L^2(0, 1)}
     \right].
\end{split}
\end{equation}
To control $J_{22}$, we point out that
\begin{equation*}
\begin{split} 
     \int_0^1   & i^n(t,y) \Big| \beta (t, y) W (t, x,y)  - \, \beta^n (t, y)
     W_{\G_n}  (t, x, y) \Big| dy\\ 
     \leq & \,
     \|    i^n (t, \cdot) \|_{L^2(0, 1)} \| [\beta W - 
     \beta^n W_{\G_n}](t,x, \cdot) \|_{L^2(0, 1)}
     \stackrel{\eqref{e:mp}}{\leq} \| [\beta W -
      \beta^n W_{\G_n}]  (t,x, \cdot) \|_{L^2(0, 1)},
\end{split}
\end{equation*}
and combining~\eqref{e:mp} with Young's Inequality we get  
\begin{equation}\label{e:j22}
   \int_0^1 J_{22} (t, x) dx \leq   \|s -  s^n  \|^2_{L^2(0, 1)} +
   \| [\beta W -
      \beta^n W_{\G_n}]  (t,\cdot,\cdot) \|^2_{L^2([0, 1]^2)}\,.    
\end{equation}
Coupling \eqref{e:smenos},~\eqref{e:j1}, \eqref{e:j2}, \eqref{e:j21}, \eqref{e:j22} we get 
\begin{equation}\begin{split}
\label{e:sms}
   \frac{d}{dt} \int_0^1 |s-s^n|^2 (t, x) dx& \leq  \| \beta \|_{\infty}
      \| W (t, \cdot, \cdot) \|_{L^2 ([0, 1]^2)} 
     \left[ \|[s-s^n ](t, \cdot)\|^2_{L^2(0, 1)}+ \| [ i - i^n](t, \cdot) \|^2_{L^2(0, 1)}
     \right]\\ &+\|[s -  s^n](t,\cdot)  \|^2_{L^2(0, 1)} +
   \| [\beta W -
      \beta^n W_{\G_n}]  (t,\cdot,\cdot) \|^2_{L^2([0, 1]^2)}.
\end{split}
\end{equation}

{\sc Step 2:} we set $v: = s+ i$, $v^n: = s^n+ i^n$ and point out that owing to~\eqref{SIR_Wn} we have $$
    \partial_t [v -v^n] = - \gamma i + \gamma^n i^n = - \gamma [v-s] + \gamma^n [v^n - s^n]
    = - \gamma [v-v^n] + \gamma [s-s^n] -v^n [\gamma- \gamma^n] + s^n [\gamma - \gamma^n].
$$
This yields 
\begin{equation*}
\begin{split}
    \partial_t [v - v^n]^2 = \underbrace{-2\gamma [v- v_n]^2}_{\leq 0} + 2 \gamma [v - v^n][s-s^n] - 2 v_n [v - v^n][\gamma - \gamma^n]
   +2  s^n [v - v^n][\gamma - \gamma^n]
\end{split}
\end{equation*}
and by integrating, recalling \eqref{e:mp}, noting that $0 \leq v, v^n \leq 1$ and using  H\"older's and Young's inequalities we get 
\begin{equation} \label{umu} \begin{split}
   \f{d}{dt}\int_0^1 [v - v^n]^2(t, x) dx &
   \leq \| \gamma\|_{\infty}\Big[ \| [v- v^n](t, \cdot) \|^2_{L^2 (0, 1)} + 
    \| [s- s^n](t, \cdot) \|^2_{L^2 (0, 1)} \Big] \\& +
   2 \| [v- v^n](t, \cdot) \|^2_{L^2 (0, 1)}  +      2 \| [\gamma- \gamma^n](t, \cdot) \|^2_{L^2 (0, 1)}.
\end{split}
\end{equation}

{\sc Step 3:} by combining~\eqref{e:sms} and~\eqref{umu} and pointing out that $\| [i-i^n](t,\cdot) \|^2_{L^2 (0, 1)} \leq 2 \big(\| [v-v^n](t,\cdot) \|^2_{L^2 (0, 1)} + \| [s-s^n](t,\cdot) \|^2_{L^2 (0, 1)}\big)$ we get
\begin{equation*}
\begin{split}
  \frac{d}{dt} \Big[  \| [s-s^n ](t, \cdot) \|^2_{L^2 (0, 1)} & + \| [v-v^n](t, \cdot) \|^2_{L^2 (0, 1)}  \Big]\\ &\leq
    \| [\beta W -
      \beta^n W_{\G_n}] (t,\cdot,\cdot) \|^2_{L^2([0,1]^2)}+  2 \| [\gamma- \gamma^n](t, \cdot) \|^2_{L^2 (0, 1)}\\ & + \| [s-s^n ](t, \cdot) \|^2_{L^2 (0, 1)} \Big[ 3 \| \beta \|_{\infty}
      \| W (t, \cdot, \cdot) \|_{L^2 ([0, 1]^2)} + \| \gamma \|_{\infty} + 1 \Big] \\ & +
     \| [v-v^n ](t, \cdot) \|^2_{L^2 (0, 1)} \Big[ 2 \| \beta \|_{\infty}
      \| W (t, \cdot, \cdot) \|_{L^2 ([0, 1]^2)} + \| \gamma \|_{\infty} + 2 \Big].
\end{split}
\end{equation*}
Owing to Gr\"onwall Lemma and to the identities $i = v-s$, $i^n=v^n - s^n$, $r=1-v$, $r^n = 1- v^n$ we eventually arrive at~\eqref{e:grolem}.
\end{proof}
\subsubsection{Conclusion of the proof of Theorem~\ref{thm:det}  {\sc (Case 1)}}\label{sss:condet}
In view of Lemma~\ref{l:gronwall}, to establish Theorem~\ref{thm:det} ({\sc Case 1}) we are left to show that the right hand side of~\eqref{e:grolem} converges to $0$ as $n \to + \infty$ uniformly in $t \in [0, T]$. 
Owing to Lemma~\ref{l:average}, we get for the initial data
\begin{equation}
\label{e:convergence2}
     \lim_{n \to +\infty}  \|s_0 -  s^n_0  \|^2_{2} =0, \quad   \lim_{n \to +\infty}  \|i_0 - i^n_0  \|^2_{2} =0, \quad 
         \lim_{n \to +\infty}  \|r_0 -  r^n_0  \|^2_{2} =0\,.
\end{equation}
Arguing analogously as in \eqref{boundl21}, we obtain 
\begin{equation} \label{boundl22}
     \| \beta^n (t, \cdot) \|_{L^2 (0, 1)} \leq  
      \| \beta(t, \cdot) \|_{L^2 (0, 1)}, \quad 
 \| \gamma^n (t, \cdot) \|_{L^2 (0, 1)} \leq  
      \| \gamma(t, \cdot) \|_{L^2 (0, 1)} \quad \text{for a.e. $t \in [0, T]$}\,.
\end{equation}
Using again \eqref{e:convl2} from Lemma~\ref{l:average}, we get 
$$
      \lim_{n \to +\infty}  \| \beta(t, \cdot)  - \beta^n 
         (t, \cdot) \|^2_{L^2(0, 1)} =0, \quad     \lim_{n \to +\infty}  \| \gamma(t, \cdot)  - \gamma^n 
         (t, \cdot) \|^2_{L^2(0, 1)} =0 \quad \text{for a.e. $t \in [0, T]$}
$$
and by virtue of the Lebesgue's Dominated Convergence Theorem together with~\eqref{boundl22} this yields 
\begin{equation}
\label{e:gamma}
          \lim_{n \to +\infty} \int_0^T  \| \beta(t, \cdot)  - \beta^n 
         (t, \cdot) \|^2_{L^2(0, 1)} dt =0, \quad \lim_{n \to +\infty} \int_0^T \| \gamma(s, \cdot)  - \gamma^n 
         (s, \cdot) \|^2_{L^2(0, 1)} ds =0\,. 
\end{equation}
Next, we use H\"older's inequality and point out that 
\begin{equation}\label{e:bw}
\begin{split}
         &\| [\beta W  -
      \beta^n  W_{\G_n}]  (t, \cdot,\cdot) \|^2_{L^2 ([0, 1]^2)}
             \leq 
      \Big(  \|[ \beta W -
      \beta^n  W](t, \cdot,\cdot) \|_{L^2 ([0, 1]^2)}   +
       \|[ \beta^n W  -
      \beta^n  W_{\G_n}] (t, \cdot,\cdot) \|_{L^2 ([0, 1]^2)} \Big)^2 \\
     &   \leq
      \Big(    \|[ \beta W -
      \beta^n  W](t, \cdot,\cdot) \|_{L^2 ([0, 1]^2)} +
       \| \beta^n (t, \cdot) \|_{L^\infty (0, 1)} \|[ W  -
      W_{\G_n}  ](t, \cdot,\cdot) \|_{L^2 [0, 1]^2} \Big)^2 \\ &
    \leq 
      2  \|[ \beta W -
      \beta^n  W](t, \cdot,\cdot) \|^2_{L^2 ([0, 1]^2)}  +
    2  \| \beta^n (t, \cdot) \|^2_{L^\infty (0, 1)} \|[ W  -
      W_{\G_n}  ](t, \cdot,\cdot) \|^2_{L^2 ([0, 1]^2)}\,.
\end{split}
\end{equation}
We now want to show that 
\begin{equation}
\label{e:beta:A}
          \lim_{n \to +\infty} \underbrace{\int_0^T \| \beta W(s, \cdot, \cdot)  - \beta^n W
         (s, \cdot, \cdot) \|^2_{L^2([0, 1]^2)} ds }_{: = H_n}=0\,.
\end{equation}
To this end, it suffices to show that, for every subsequence $\{ H_{n_j} \}_{j \in \mathbb N}$, there is a further subsubsequence $\{ H_{n_{j_h}} \}_{h \in \mathbb N}$
such that $\lim_{h \to + \infty} H_{n_{j_h}} =0.$ We therefore fix an arbitrary subsubsequence $\{ H_{n_j} \}_{j \in \mathbb N}$ and notice that from~\eqref{e:settembre} it follows that there is a subsubsequence $\beta^{n_{j_h}}(t, x)$  that converges to $\beta(t, x)$ for a.e. $(t, x) \in [0, T] \times [0, 1] $.  This implies that $W \beta^{n_{j_h}}(t, x)$ converges to $W \beta(t, x)$ for a.e. $(t, x) \in [0, T] \times [0, 1] $. Hence,  the bound 
$$
    |W [ \beta^{n_{j_h}} - \beta] |^2  \leq |W|^2 [\| \beta^{n_{j_h}} \|_{\infty} + \| \beta \|_{\infty} ]^2  \leq 4   |W|^2  \| \beta \|_{\infty}^2 \quad \text{a.e. on $[0, T] \times [0, 1]^2$}
$$
together with the Lebesgue's Dominated Convergence Theorem  yields $\lim_{h \to + \infty} H_{n_{j_h}} =0$ and concludes the proof of~\eqref{e:beta:A}. 
Owing to~\eqref{e:convl2}, \eqref{boundl21} and to the Lebesgue's Dominated Convergence Theorem we also have 
\begin{equation}
\label{e:beta}
        \lim_{n \to +\infty} \int_0^T \| W(s, \cdot, \cdot)  - W_{\G_n} 
         (s, \cdot, \cdot) \|^2_{L^2([0, 1]^2)} ds =0.
\end{equation}
Substituting \eqref{e:beta:A} and \eqref{e:beta} into~\eqref{e:bw}, using~\eqref{e:gamma} and~\eqref{e:convergence2} and recalling~\eqref{e:grolem} we eventually arrive at~\eqref{e:cisiamo}.
\subsection{Proof of Theorem~\ref{thm:det}  ({\sc Case 2})}\label{ss:pcdet}
The proof follows the same argument as  for ({\sc Case 1})  ,
so we only  sketch and highlight the points where there are differences. 
We subtract the first line of~\eqref{SIR_Wn} from the first line of~\eqref{SIR_W} and multiply the result by~$\mathrm{sign}[s-s^n]$. Next, we integrate in space and arguing as in {\sc Step 1} of the proof of Lemma~\ref{l:gronwall} we get 
\begin{equation}
\label{gronwall2}
  \f{d}{dt} \int_0^1 |s - s^n|(t, x) dx \leq \int_0^1 \int_0^1 \beta (t, y) W(t, x, y) |i - i^n|(t, y) dydx +
   \int_0^1 \int_0^1 |\beta W -  \beta^n W_{\G_n}|(t, x, y) dy dx.
\end{equation}
From the symmetry property $W(t, x, y) = W(t, y, x)$ we immediately get
\[
\begin{split}
   \int_0^1 \int_0^1 \beta (t, y) W(t, x, y) |i - i^n|(t, y) dxdy &\,\leq \| \beta \|_{\infty}
     \int_0^1  |i - i^n|(t, y) \int_0^1  W(t, x, y) dx dy \\
     &\,\stackrel{\eqref{bound_W}}{\leq} K_d\|\beta\|_\infty \| [i-i^n](t,\cdot)\|_{L^1(0, 1)}.
\end{split}
\]
By plugging the above inequality in~\eqref{gronwall2} and then arguing as in the proof of Lemma~\ref{l:gronwall} we arrive at 
\begin{equation*}\begin{split}
  \|[s  -  s^n] &(t, \cdot)  \|_{L^1(0, 1)} 
     \! \! +  \|[i -  i^n] (t, \cdot)  \|_{L^1(0, 1)} + 
     \|[r -  r^n] (t, \cdot)  \|_{L^1(0, 1)}
      \\& \leq \! 
     D \Big(
       \|s_0 - s^n_0  \|_{1} \! \!
     +  \|i_0 -  i^n_0  \|_{1} \! \!+
     \|r_0 -  r^n_0   \|_{1} \Big)
      \\ & \quad + 
      D \displaystyle{\int_0^t
       \Big( \| [\beta W -
      \beta^n W_{\G_n} ] (s,\cdot, \cdot) \|_{L^1([0, 1]^2)}
     +  \| [\gamma - \gamma^n] (s, \cdot) \|_{L^1(0, 1)}
     \Big) ds} \quad \text{for every $t \in [0, T]$},
\end{split}
\end{equation*}
for a suitable constant $D$  only depending on $T, K_d, \| \beta \|_{\infty}$ and 
 $\| \gamma \|_{\infty}$. Hence, we need to ensure convergence of the right hand side of the above expression.
 We first point out that
\begin{equation*}
  \| [\beta W -
      \beta^n W_{\G_n} ] (s,\cdot, \cdot) \|_{L^1([0, 1]^2)}
\leq  \| [\beta -\beta^n]W (s,\cdot, \cdot) \|_{L^1([0, 1]^2)}+ 
    \| [\beta^n [W -
      W_{\G_n} ] (s,\cdot, \cdot) \|_{L^1([0, 1]^2)}\,.
\end{equation*}
The first term in the above sum can be further estimated by 
$$
  \| [\beta -\beta^n]W (s,\cdot, \cdot) \|_{L^1([0, 1]^2)} =
   \int_0^1 |\beta - \beta^n|(t, y) \int_0^1 W(s, x, y) dx dy 
   \stackrel{\eqref{bound_W}}{\leq}
  K_d \| [\beta - \beta^n](s, \cdot) \|_{L^1(0, 1)},
$$
while to control the second term it suffices to note that, by construction, 
$\| \beta^n \|_{\infty}\leq \| \beta \|_{\infty} $ and this implies 
$\| [\beta^n (W -
      W_{\G_n} )] (s,\cdot, \cdot) \|_{L^1([0, 1]^2)} \leq \| \beta \|_{\infty}  \| [W -
      W_{\G_n} ] (s,\cdot, \cdot) \|_{L^1([0, 1]^2)}$. 
      
      The rest of the argument  is basically the same as for {\sc Case 1} in \S\ref{sss:condet}, the main difference is that to control $\| [W -
      W_{\G_n} ] (s,\cdot, \cdot) \|_{L^1([0, 1]^2)}$ we use~\eqref{e:convl1} from Lemma~\ref{l:average} instead of~\eqref{e:convl2}. 
      
\begin{remark}
    As pointed out in Remark~\ref{rem:si_Linfty}, from the point of view of applications it would be also reasonable to consider the SIR model~\eqref{SIR'_G}. As a matter of fact, the proof of Theorem~\ref{thm:det} does extend to~\eqref{SIR'_G}. More precisely, under the same assumptions as in Theorem~\ref{thm:det}, the solution $(s^n, i^n, r^n)$ of \eqref{SIR'_G} satisfies~\eqref{e:cisiamo}, provided $s,i,r:[0,T]\times[0,1]\to[0,1]$ is the  distributional solution of 
 	\[
 	\begin{cases}
 	\displaystyle{\partial_t s(t,x)=-s(t,x)\int_0^1\beta(t,y)W(t,x,y)i(t,y)\,dy-\beta(t,x)s(t,x)i(t,x)} & \\
 	\displaystyle{\partial_t i(t,x)=s(t,x)\int_0^1\beta(t,y)W(t,x,y)i(t,y)\,dy+   
\beta(t,x)s(t,x)i(t,x)-\gamma(t,x)i(t,x)} & \qquad\qquad x\in[0,1]\\
 	\partial_tr(t,x)=\gamma(t,x)i(t,x) \phantom{\displaystyle{\int}}
 	\end{cases}
 	\]
coupled with the initial condition~\eqref{e:idSIR_W}. 
 \end{remark}  
 
 \section{Random approximation: convergence analysis of the random samplings}
\label{sec:proof2}
In this section we present the convergence analysis for the random approximation as stated in Theorem \ref{thm:ran}.

The proof is inspired by the argument used in \cite{Med19} for Kuramoto systems. In short, at each step of the approximation, we will introduce an auxiliary deterministic SIR system on an averaged graph and split the total error of the approximation in two components: a random error (between the original random model and the auxiliary averaged one) and a deterministic error (between the averaged model and the limit one). 

The exposition in \S\ref{sec:proof2} is organized as follows. In \S\ref{subsec:prel} we review some preliminary probability results. In \S\ref{subsec:average} we pass from the random to the averaged model.
The proof of Theorem \ref{thm:ran} is carried out in \S\ref{ss:ptran} for {\sc Case 1}  and in  \S\ref{ss:pcran} for {\sc Case 2}.
 
\subsection{Preliminary results}
\label{subsec:prel}
We begin by stating (without proof) three classical results in Probability Theory. The first one is the Borel--Cantelli Lemma (see for instance \cite[Chapter 3]{C}).
\begin{lemma}[Borel-Cantelli]
Let $\{ \mathcal{E}_n \}_{n \in \nn}$ be a sequence of events in a probability space $(\Omega, \mathcal F, \mathbb P)$ and set $\mathcal{E}:=\limsup_{n\to+\infty}\mathcal{E}_n = \bigcap_{k =1}^\infty \bigcup_{n \ge k} \mathcal{E}_n$. If $\sum_{n\in\nn}\mathbb{P}(\mathcal{E}_n)<\infty$, then $\mathbb{P}\left(\mathcal{E}\right)=0$.
\end{lemma}
The second tool is Chebyshev's inequality: let $X$ be a random variable  on a probability space $(\Omega, \mathcal F, \mathbb P)$ with bounded second moment. Then,
\begin{equation}
\label{e:cheb}
       \mathbb P \left(\left|X-\mathbb{E}X\right|\ge t\right) \leq \frac{{\rm Var}(X)}{t^2} \quad \text{for every }t>0, 
\end{equation}
where ${\rm Var}(X) : = \mathbb{E}\Big( \big( X - \mathbb{E}(X) \big)^2\Big)= \mathbb{E}(X^2) - \mathbb{E}(X)^2$. 

The third tool is Hoeffding's inequality: let  $X_1, \dots, X_m$ be independent randon variables  on a probability space $(\Omega, \mathcal F, \mathbb P)$ such that   $a \leq X_i \leq b$ almost surely for every $i=1, \dots, m$, for some $a,b\in\rr$. Then, for every $t>0$  it holds
\begin{equation}\begin{split}
\label{e:hoeff}
        & \mathbb{P} \left( \sum_{i=1}^m X_i -  \sum_{i=1}^m  \mathbb{E}(X_i)  \ge t \right)  
         \leq \exp\left[ - \frac{2 t^2}{ m(b-a)^2}\right], \quad \text{for every $t>0$,}\\&
         \mathbb{P} \left(\left|  \frac{1}{m}\sum_{i=1}^m X_i -  \frac{1}{m}\sum_{i=1}^m  \mathbb{E}(X_i)  \right| \ge t \right)  
         \leq \exp\left[ - \frac{2 mt^2}{ (b-a)^2}\right] \quad \text{for every $t>0$.}
\end{split}
\end{equation}
We now recall that, for $u^n$ as in~\eqref{uG->u01}, the $L^{2}$ and $L^{1}$ discrete norms are defined by 
\begin{equation}
\label{e:norma2}
\|u^n\|_{2}=\left(\f1n\sum_{j=1}^n |u^n_j|^2\right)^{1/2} \! \! \! \! \! \!\! , \qquad 
 \|u^n\|_{1}=\f1n\sum_{j=1}^n |u^n_j|,
\end{equation}
and  the discrete H\"older's inequality reads
\begin{equation}
\label{holderd}
     \sum_{i=1}^n x_i y_i \leq \left(  \sum_{i=1}^n x_i^2 \right)^{1/2} \left(  \sum_{i=1}^n y_i^2 \right)^{1/2} 
      \quad \text{for every $x_1, \dots, x_n, y_1, \dots, y_n \in \rr$.}
\end{equation}

Furthermore, note that, adapting the argument given in Lemma \ref{l:zerouno} to prove \eqref{e:01SIR_G}, if $A^n_{jk}$ is the random variable~\eqref{e:adjr1} and $(s^n, i^n, r^n)$ is the solution of the Cauchy problem~\eqref{SIR_G}, \eqref{e:idSIR_G}, written according to the notation~\eqref{uG->u01}, then 
\begin{equation}
\label{ran:bd}
   0 \leq s^n_j, i^n_j, r^n_j \leq 1, \quad 
   s^n_j+ i^n_j+ r^n_j = 1
 \quad  \text{almost surely and for a.e. $(t, x) \in [0, T] \times [0, 1]$}. 
\end{equation}

\subsection{From a random to a deterministic model: the SIR system on the averaged graph}
\label{subsec:average}
For any given $n \in \nn$ and starting from the sparse $W$-random graph $\G_n=n^\alpha\G(n,W,n^{-\alpha})$ (or $\G_n=\G(n,W)$) we consider the averaged graph, denoted by $\bar{\G}_n$, defined by setting
\begin{equation}
\label{avg_A}
     \bar A^n_{jk} : = \mathbb{E}(A^n_{jk}) \stackrel{\eqref{e:adjr1}}{=} n^\alpha < \widehat W^n>_{I^n_j \times I^n_k}, \quad \text{where $ \widehat W^n = \min \{ 1, n^{-\alpha} W \}\;.$ }
\end{equation} 
Using always the notation~\eqref{uG->u01} we denote by $(\bar s^n, \bar i^n, \bar r^n)$ the solution of the SIR system on $\bar{\G}_n$
\begin{equation}
\label{SIR_avg}
\begin{cases}
\displaystyle{\f{d\bar{s}_j^n(t)}{dt}=-\bar{s}_j^n(t)\f1n\sum_{k=1}^n \beta_k^n(t)\bar{A}_{jk}^n \bar{i}_k^n(t)}, & \\
\displaystyle{\f{d\bar{i}_j^n(t)}{dt}=\bar{s}_j^n(t)\f1n\sum_{k=1}^n\beta_k^n(t)
\bar{A}_{jk}^n\bar{i}_k^n(t)-\gamma_j^n(t)\bar
i_j^n(t) }& \qquad\qquad\qquad j=1,\,\dots,\,n,\\
\displaystyle{\f{d\bar{r}_j^n(t)}{dt}=\gamma_j^n(t)\bar i_j^n(t)\,, }
\end{cases}
\end{equation}
coupled with the initial conditions~\eqref{e:idSIR_G}. Note that the coefficients $\beta^n_j$ and $\gamma^n_j$ and the initial data $s^n_{j,0}, i^n_{j, 0}, r^n_{j, 0}$ are as in~\eqref{e:betagamma}. 
Note that the above system  \eqref{SIR_avg} is indeed a {\it deterministic} approximation to \eqref{SIR_W}.

From the definition of the averaged graph $\bar{\G_n}$, and depending on the regularity of $W$, we can derive the following estimates that will be useful for our analysis:
\begin{description}
\item[$\bullet$ {\sc Case 1:} ] if $W \in L^2 ([0, 1]^2;\rr_+)$ then 
\begin{equation}
\label{normal2d}\begin{split}
      \sum_{j, k=1}^n  \f{[\bar A^n_{jk}]^2}{n^2} & \stackrel{\eqref{avg_A}}{=}
       \sum_{j, k=1}^n \frac{1}{n^2 } \left( n^2 \int_{I^n_j} \int_{I^n_k} n^\alpha \widehat W^n (x, y) dx dy\right)^2 \\ &
         \stackrel{\text{Jensen}}{\leq}
         \sum_{j, k=1}^n \frac{1}{n^2}  n^{2} \int_{I^n_j} \int_{I^n_k} [ n^\alpha \widehat W^n (x, y)]^2 dx dy \leq 
         \sum_{j, k=1}^n  \int_{I^n_j} \int_{I^n_k} [  W(x, y)]^2 dx dy = \| W \|^2_2\,;
\end{split}
\end{equation}
\item[$\bullet$ {\sc Case 2:} ] if $W\in L^1 ([0, 1]^2;\rr_+)$ satisfies~\eqref{bound_W2} then 
\begin{equation}\begin{split}
\label{e:grado}
    \frac{1}{n}  \sum_{j=1}^n \bar A^n_{jk} & \stackrel{\eqref{avg_A}}{=}
       \frac{1}{n} \sum_{j=1}^n n^\alpha < \widehat W^n>_{I^n_j \times I^n_k}
      \leq    \frac{1}{n} \sum_{j=1}^n <  W>_{I^n_j \times I^n_k}
      =   \frac{1}{n}\sum_{j=1}^n n^2 \int_{I^n_j} \int_{I^n_k} W(x, y) dx dy  \\ &
    =  n \int_{I^n_j} \int_{0}^1 W(x, y) dx dy \stackrel{\eqref{bound_W2}}{\leq} K_{a}\,.
\end{split}
\end{equation}
\end{description}
Note furthermore that, by construction,
\begin{equation}
\label{media0}
       \mathbb{E} (   \bar A^n_{jk} -A^n_{jk} ) \stackrel{\eqref{avg_A}}{=} 0.
\end{equation}

\subsection{Proof of Theorem~\ref{thm:ran} {\sc (Case 1)}}\label{ss:ptran}
In what follows we always assume $W\in L^{2}([0,1]^2; \mathbb{R}_+)$.
We start by proving the following lemma, which provides an estimate (with probability one) on the difference between the random approximation $( s^n,  i^n,  r^n)$  and the deterministic approximation $(\bar s^n, \bar i^n, \bar r^n)$ solution of the associated averaged model.
\begin{lemma}
	\label{lem:G_to_average}
     Under the same assumptions as in the statement of Theorem~\ref{thm:ran} {\sc (Case 1)}, let  	
    $\bar s^n, \bar i^n, \bar r^n:[0,T]\times [0, 1]\to \rr$ be the solution of the Cauchy problem \eqref{SIR_avg},~\eqref{e:idSIR_G}. Then 
	\begin{equation} \label{e:diff}
	\lim_{n\to+\infty} \Big[ \| [s^n  - \bar s^n](t, \cdot) \|_{L^2(0, 1)}
     + \| [i^n  - \bar i^n](t, \cdot) \|_{L^2(0, 1)}+ \| [r^n  - \bar r^n](t, \cdot) \|_{L^2(0, 1)} \Big]=0
	\end{equation}
$\mathbb P$-almost surely and for a.e. $t \in [0, T]$.
\end{lemma}
\begin{proof}
We proceed according to the following steps. 

\smallskip
{\sc Step 1:} we subtract~\eqref{SIR_G} from~\eqref{SIR_avg}, multiply the result by $2 [\bar s^n_j - s^n_j]$, use discrete H\"older \eqref{holderd} to obtain 
\begin{equation*}
\begin{split}
&\f{d}{dt} [\bar s^n_j - s^n_j]^2 = - \underbrace{\f{2[\bar s^n_j - s^n_j]^2}{n} \sum_{k=1}^n \beta^n_k \bar A^n_{jk} \bar i_k^n}_{\ge 0}  
-\f{2s_j^n [\bar s^n_j - s^n_j]}{n} \sum_{k=1}^n  \beta^n_k \Big[ \bar A^n_{jk} \bar i^n_k - 
 A^n_{jk}  i^n_k \Big] \\ & \stackrel{\eqref{linftybk},\eqref{ran:bd}}{\leq} 
2 \| \beta \|_{\infty} \f{|\bar s^n_j - s^n_j|}{n} \sum_{k=1}^n  \bar A^n_{jk} | \bar i^n_k - i^n_k|
+ \f{2s_j^n [\bar s^n_j - s^n_j]}{n} \sum_{k=1}^n  \beta^n_k i^n_k \Big[ \bar A^n_{jk}  - 
 A^n_{jk}   \Big] \\ &
\stackrel{\eqref{holderd},\eqref{ran:bd}}{\leq}
2 \| \beta \|_{\infty} \f{|\bar s^n_j - s^n_j|}{n} 
\left( \sum_{k=1}^n  [\bar A^n_{jk}]^2 \right)^{1/2}
 \left( \sum_{k=1}^n   | \bar i^n_k - i^n_k|^2 \right)^{1/2}+2 [\bar s^n_j - s^n_j] Z^n_j (t)\,,
\end{split}
\end{equation*}
where we have set
\begin{equation} \label{e:zetaenne}
    Z^n_j (t) : =  \f{s^n_j (t)}{n} \sum_{k=1}^n  \beta^n_k (t) i^n_k(t) \Big[ \bar A^n_{jk}  - 
 A^n_{jk}   \Big].
\end{equation}
Hence, summing over all nodes and using the discrete H\"older Inequality~\eqref{holderd}, we have 
\begin{equation} \label{l2ds}
\begin{split}
& \f{d}{dt}\sum_{j=1}^n [\bar s_j^n - s_j^n]^2  \leq 
2 \| \beta \|_{\infty} \left( \sum_{k=1}^n   | \bar i_k^n - i_k^n|^2 \right)^{1/2}
\sum_{j=1}^n|\bar s^n_j - s_j^n|
\left( \sum_{k=1}^n  \f{[\bar A^n_{jk}]^2}{n^2} \right)^{1/2} +
2 \sum_{j=1}^n
 [\bar s^n_j - s_j^n] Z^n_j \\
& \stackrel{\eqref{holderd}}{\leq}
2 \| \beta \|_{\infty} \left( \sum_{k=1}^n   | \bar i_k^n - i_k^n|^2 \right)^{1/2}   \! \!  \! 
\left( \sum_{j=1}^n|\bar s^n_j - s_j^n|^2\right)^{1/2} \! \! \! 
\underbrace{\left( \sum_{j, k=1}^n \f{[\bar A^n_{jk}]^2}{n^2} \right)^{1/2}}_{\leq \| W\|_{2} \; \text{by~\eqref{normal2d}}}     \! \!\!  \! \! \! \! \! +
2 \left( \sum_{j=1}^n
 [\bar s^n_j - s_j^n]^2 \right)^{1/2} \! \! \!   \! 
 \left( \sum_{j=1}^n (Z^n_j)^2\right)^{1/2}\\ & 
 \stackrel{\text{Young}}{\leq}
\| \beta \|_{\infty} \| W\|_{2} \sum_{j=1}^n|\bar i^n_j - i^n_j|^2 +
\| \beta \|_{\infty}  \Big[\| W\|_{2} + 1\Big]  \sum_{j=1}^n|\bar s^n_j - s_j^n|^2 
+  \sum_{j=1}^n  (Z^n_j)^2\,.
\end{split} 
\end{equation}
Next, we use an argument similar to the one in the proof of Proposition \ref{p:uniex} in \S\ref{ss:puni}. We sum the first two equations in~\eqref{SIR_G}, subtract it from the sum of the first two equations in~\eqref{SIR_avg} and set $\bar v^n_j : = \bar s^n_j + \bar i^n_j$, $v^n_j : =  s^n_j +  i^n_j$ to get
\begin{equation}\label{derdif}
   \frac{d}{dt} [\bar v^n_j - v^n_j] = - \gamma^n_j [\bar v^n_j - v^n_j] + \gamma^n_j [\bar s^n_j - s^n_j]\,.
\end{equation}
Multiplying the above equation by $2[\bar v^n_j - v^n_j]$ we find 
\begin{equation*}
\begin{split}
          \frac{d}{dt} [\bar v^n_j - v^n_j]^2 & = - \underbrace{2 \gamma^n_j [\bar v^n_j - v^n_j]^2}_{\ge 0}
          +2  \gamma^n_j [\bar s^n_j - s^n_j] [\bar v^n_j - v^n_j]
         \stackrel{\eqref{linftybk},\text{Young}}{\leq}
         \| \gamma \|_{\infty} \Big[  [\bar s^n_j - s^n_j]^2 +  [\bar v^n_j - v^n_j]^2 \Big]\;,
\end{split}
\end{equation*}
which in turn entails
\begin{equation}\label{l2dv}
\begin{split}
          \frac{d}{dt}\sum_{j=1}^n [\bar v^n_j - v^n_j]^2 \leq 
         \| \gamma \|_{\infty} \left[ \sum_{j=1}^n [\bar s^n_j - s^n_j]^2 +  \sum_{j=1}^n
         [\bar v^n_j - v^n_j]^2 \right] .
\end{split}
\end{equation}

{\sc Step 2:} we combine the estimates~\eqref{l2ds} and~\eqref{l2dv}, and note that
$$
     \sum_{k=1}^n   | \bar i^n_k - i^n_k|^2 \leq  
     2 \sum_{k=1}^n   | \bar v^n_k - v^n_k|^2 + 2 \sum_{k=1}^n   | \bar s^n_k - s^n_k|^2\,.
$$
Recalling now the expression of the $L^2$ norm~\eqref{e:norma2} and applying Gr\"onwall Lemma,  we eventually arrive at
\begin{equation}\label{groavg}
\begin{split}
      \| [\bar v^n - v^n](t, \cdot) \|^2_{L^2(0, 1)} & +  \| [\bar s^n - s^n](t, \cdot) \|^2_{L^2(0, 1)} 
     \\ & \leq \exp[ \bar B t] \frac{1}{n} \int_0^t \exp[- \bar B \tau ]  \sum_{j=1}^n    (Z^n_j)^2(\tau) d \tau 
     \quad \text{almost surely and for every $t \in [0, T],$}
\end{split}
\end{equation}
for a suitable constant $\bar B$ only depending on $\| \beta \|_{\infty}, \| \gamma \|_{\infty}$ and 
$\| W \|_2$. Owing to the equalities $\bar v^n_j : = \bar s^n_j + \bar i^n_j$, $v^n_j : =  s^n_j +  i^n_j$, $\bar r^n_j = 1 - \bar v^n_j$, $r^n_j = 1 - v^n_j$, to establish~\eqref{e:diff} it suffices to show that the right hand side of~\eqref{groavg} vanishes in the $n \to + \infty$ limit, almost surely. To this end we notice that
\begin{equation} 
	\label{qtc} 
	\begin{split}
    \frac{1}{n}\int_0^T& \exp[-\bar B s] \sum_{j=1}^n  (Z^n_j)^2 (s) ds \stackrel{\eqref{e:zetaenne}}{=}
     \frac{1}{n^3} \sum_{j=1}^n   \sum_{k=1}^n | \bar A^n_{jk} -A^n_{jk} |^2 \int_0^T \exp[-\bar B\tau]  [s^n_j (\tau)]^2
     [\beta^n_k i^n_k]^2(\tau)  d \tau \\
&    + \frac{2}{n^3} \sum_{j=1}^n 
        \sum_{\substack{k, \ell=1\\ k \neq \ell}}^n
     ( \bar A^n_{jk} -A^n_{jk} ) ( \bar A^n_{j \ell} -A^n_{j \ell} )
      \int_0^T  [s^n_j (\tau)]^2
     [\beta^n_k  i^n_k \beta^n_\ell i^n_\ell] (\tau) \exp[-\bar B\tau]d \tau\,.
   \end{split}
\end{equation}
To control the first term in the above sum, we point out that almost surely we have
\[
\begin{split}
   0 \leq&\, \frac{1}{n^3} \int_0^T \exp[-\bar B\tau] \sum_{j=1}^n  [s^n_j (\tau)]^2\Big[  \sum_{k=1}^n [\beta^n_k i^n_k]^2(\tau)| \bar A^n_{jk} -A^n_{jk} |^2 \Big] d \tau\\ &\stackrel{\eqref{ran:bd},\eqref{linftybk}}{\leq}
   \| \beta \|_{\infty}  \frac{1}{n^3}  \sum_{j, k=1}^n | \bar A^n_{jk} -A^n_{jk} |^2   \int_0^T \exp[-\bar B\tau] d \tau. 
\end{split}
\]
Note that $\{| \bar A^n_{jk} -A^n_{jk} |^2\}_{j,k=1, \dots, n}$ are independent random variables and, thanks to~\eqref{e:adjr1} and \eqref{avg_A}, satisfy
$$
    0 \leq | \bar A^n_{jk} -A^n_{jk} |^2 \leq 
    n^{2 \alpha}  \qquad \text{almost surely}
$$
together with
\begin{equation} \begin{split} \label{e:media}
    \mathbb{E} ( | \bar A^n_{jk} -A^n_{jk} |^2) & =
    \mathbb{E} ([A^n_{jk}]^2 ) - [\mathbb{E}(A^n_{jk})]^2 \stackrel{\eqref{e:adjr1}}{=} 
    n^{2 \alpha} < \widehat W^n >_{I^n_j \times I^n_j} [1-  < \widehat W^n >_{I^n_j \times I^n_j}] 
    \leq n^{2 \alpha}.
\end{split} 
\end{equation}
Hoeffding's inequality~\eqref{e:hoeff} then yields
\begin{equation} \label{hoeff1}
    \mathbb{P} \left( \sum_{j, k=1}^n | \bar A^n_{jk} -A^n_{jk} |^2 - 
    \sum_{j, k=1}^n \mathbb{E}( | \bar A^n_{jk} -A^n_{jk} |^2) \ge t
 \right) \leq 
  \exp\left( - 2 \left[\f{ t}{n^{2 \alpha +1}} \right]^2 \right)\;.
\end{equation}
Setting $t : = n^\beta$ with $\beta \in \,] 1 + 2\alpha,  3[$ in \eqref{hoeff1}, we apply Borel-Cantelli Lemma to the event 
$$
   \mathcal E_n : =  \left\{\sum_{j, k=1}^n | \bar A^n_{jk} -A^n_{jk} |^2 
  \ge \sum_{j, k=1}^n \mathbb{E}( | \bar A^n_{jk} -A^n_{jk} |^2) + n^\beta\right\}\;,
$$
to obtain
\begin{equation*}
\begin{split}
    \lim_{n \to +\infty} & n^{-3} \sum_{j, k=1}^n | \bar A^n_{jk} -A^n_{jk} |^2 
  \leq \lim_{n \to + \infty}
      n^{-3} \sum_{j, k=1}^n \mathbb{E}( | \bar A^n_{jk} -A^n_{jk} |^2) + \lim_{n \to + \infty} n^{\beta-3} \\
   & \stackrel{\eqref{e:media}}{\leq} \lim_{n \to + \infty}
      n^{2 \alpha -1} +  \lim_{n \to + \infty} n^{\beta-3} =0
\qquad\text{almost surely}.
\end{split}
\end{equation*}
We now turn our attention to the second term in~\eqref{qtc}. We first observe that, owing to \eqref{ran:bd} and \eqref{linftybk}, we have 
$$
   0 \leq  \int_0^T  [s^n_j (\tau)]^2
     [\beta^n_k  i^n_k \beta^n_\ell i^n_\ell] (\tau) \exp[-\bar B\tau] 
   d \tau \leq \| \beta \|_{\infty} \int_0^T \exp[-\bar B\tau] 
   d \tau \qquad \text{almost surely}.
$$
Next, notice that for $k  \neq \ell$ the random variables
$  \bar A^n_{jk} -A^n_{jk} $ and $  \bar A^n_{j \ell} -A^n_{j \ell} $ are independent, which together with the linearity of the expectation implies 
\begin{equation*} 
   \mathbb{E}  \left(\sum_{j=1}^n  \; \sum_{k, \ell=1, k \neq \ell}^n
     ( \bar A^n_{jk} -A^n_{jk} )(  \bar A^n_{j \ell} -A^n_{j \ell} ) \right) =
     \sum_{j=1}^n \; \sum_{k, \ell=1, k \neq \ell}^n \mathbb{E} \big(  \bar A^n_{jk} -A^n_{jk}  \big)
 \mathbb{E} \big(    \bar A^n_{j \ell} -A^n_{j \ell} 
    \big) \stackrel{\eqref{media0}}{=} 0\;.
\end{equation*}
Hence, Chebyshev's inequality~\eqref{e:cheb} yields
\begin{equation} \label{e:cheb2}
    \mathbb{P} \left(\sum_{j=1}^n  \; \sum_{k, \ell=1, k \neq \ell}^n
     ( \bar A^n_{jk} -A^n_{jk} )(\bar A^n_{j \ell} -A^n_{j \ell} ) \ge t  \right) \leq  
   \frac{1}{t^2} \mathbb{E} \left( \left[\sum_{j=1}^n  \; \sum_{k, \ell=1, k \neq \ell}^n
     ( \bar A^n_{jk} -A^n_{jk} )(  \bar A^n_{j \ell} -A^n_{j \ell} )\right]^2\right).
\end{equation}
The linearity of the expectation allows us to write
\begin{equation*}
\begin{split}
          \mathbb{E}& \left( \left[\sum_{j=1}^n  \; \sum_{k, \ell=1 k \neq \ell}^n
     [ \bar A^n_{jk} -A^n_{jk} ]  [ \bar A^n_{j \ell} -A^n_{j \ell} ] \right]^2\right)=
          \sum_{j=1}^n  \; \sum_{\substack{k, \ell=1\\ k \neq \ell}}^n 
      \mathbb{E}([\bar A^n_{jk} -A^n_{jk} ]^2 [ \bar A^n_{j \ell} -A^n_{j \ell} ]^2) \\ & +
     2 \sum_{i, j=1}^n  \; \sum_{\substack{k, \ell, h, m=1\\ k, \ell, h, m \in \mathcal I} }^n 
      \mathbb{E}(  [ \bar A^n_{jk} -A^n_{jk} ][ \bar A^n_{j \ell} -A^n_{j \ell} ]
       [ \bar A^n_{ih} -A^n_{ih} ][  \bar A^n_{i m} -A^n_{j m} ] )
\end{split}
\end{equation*}
with the set of indices $ \mathcal I := \big\{ k, \ell, h, m=1, \dots, n: k \neq \ell, h \neq m, [k \neq h \lor  \ell \neq m] \land 
     [k \neq m \lor  \ell \neq h] 
     \big\}$. Observe that the independence of $\bar A^n_{jk} -A^n_{jk}$ and \eqref{media0} guarantee that all the terms in the second sum above vanish. Precisely, if $k \neq \ell, h \neq m, k \neq h, k \neq m$
\begin{equation*}
    \mathbb{E}(( \bar A^n_{jk} -A^n_{jk} )(\bar A^n_{j \ell} -A^n_{j \ell} ) 
       ( \bar A^n_{ih} -A^n_{ih} )( \bar A^n_{i m} -A^n_{i m} ) )   \! =  \! 
    \mathbb{E}( \bar A^n_{jk} -A^n_{jk} )\mathbb{E} (\bar A^n_{j \ell} -A^n_{j \ell} ) 
       \mathbb{E} ( \bar A^n_{ih} -A^n_{ih} )  \mathbb{E} ( \bar A^n_{i m} -A^n_{i m} ) 
   \!  \stackrel{\eqref{media0}}{=}   \! 0,
\end{equation*}
and arguing similarly one finds 
\begin{equation} \label{e:zeroI}
    \mathbb{E}(  ( \bar A^n_{jk} -A^n_{jk} )(\bar A^n_{j \ell} -A^n_{j \ell} ) 
       ( \bar A^n_{ih} -A^n_{ih} )( \bar A^n_{i m} -A^n_{i m} ) )=0 \qquad \text{for every $(k, \ell, h, m) \in \mathcal I$}.
\end{equation}
As for the terms in the first sum, since $k \neq \ell$, the independence together with  \eqref{e:media} gives
\begin{equation} \label{mediaq}
      \mathbb{E}([ \bar A^n_{jk} -A^n_{jk} ]^2 [\bar A^n_{j \ell} -A^n_{j \ell} ]^2) =
       \mathbb{E}([ \bar A^n_{jk} -A^n_{jk} ]^2)\mathbb{E} ( [ \bar A^n_{j \ell} -A^n_{j \ell} ]^2) 
     \stackrel{\eqref{e:media}}{\leq} n^{4 \alpha}. 
\end{equation}
Plugging~\eqref{mediaq} and~\eqref{e:zeroI} into~\eqref{e:cheb2} we arrive at 
$$ 
   \mathbb{P} \left( n^{-3} \sum_{j=1}^n  \; \sum_{k, \ell=1, k \neq \ell}^n
     | \bar A^n_{jk} -A^n_{jk} |  | \bar A^n_{j \ell} -A^n_{j \ell} | \ge t  \right) \leq  
   \frac{1}{n^{6} t^2} n^{4 \alpha + 3 } = \frac{1}{t^2} n^{4 \alpha -3 }\,.
$$
By setting $t : = n^{\sigma}$, with $\sigma \in\, ]2 (\alpha -1), 0[$,  the right hand side of the above expression boils down to $n^{4 \alpha -3 -2 \sigma}.$ Then, since $4 \alpha -3 -2 \sigma < -1$  we can apply the Borel-Cantelli Lemma and conclude that 
$$
    \lim_{n \to + \infty} n^{-3} \sum_{j=1}^n  \; \sum_{k, \ell=1, k \neq \ell}^n
     | \bar A^n_{jk} -A^n_{jk} |  | \bar A^n_{j \ell} -A^n_{j \ell} |
    \leq   \lim_{n \to + \infty} n^{\sigma} =0 \qquad \qquad\text{almost surely.}
$$
This eventually implies that the right hand side of~\eqref{groavg} vanishes almost surely 
in the $n \to + \infty$ limit and completes the proof of Lemma~\ref{lem:G_to_average}. 
\end{proof}
We consider the splitting of the error
$$
s-s^n=(s-\bar s^{n}) +(\bar s^{n} -s^{n}) \quad  i-i^n=(i-\bar i^{n}) +(\bar i^{n} -i^{n}) \qquad r-r^n=(r-\bar r^{n}) +(\bar r^{n} -r^{n} )\;.
$$
Hence, by virtue of triangle inequality and  Lemma~\ref{lem:G_to_average}, to conclude the proof of Theorem \ref{thm:ran} it suffices to show that
\begin{equation*}
	\lim_{n\to+\infty} \Big[ \| [s  - \bar s^n](t, \cdot) \|_{L^2(0, 1)}
     + \| [i  - \bar i^n](t, \cdot) \|_{L^2(0, 1)}+ \| [r  - \bar r^n](t, \cdot) \|_{L^2(0, 1)} \Big]=0
     \quad \text{for every $t \in [0, T]$}.
	\end{equation*}
Observe that $(\bar s^n, \bar i^n, \bar r^n)$ is a deterministic approximation (with a particular sampling of the graphon $W$) of \eqref{SIR_W} and in particular 	the proof of Lemma~\ref{l:gronwall} works  for $(\bar s^n, \bar i^n, \bar r^n)$ provided we replace $W_{\G_n}$ by the step graphon $W_{\bar \G_n}$ of the averaged graph $\bar \G_n$ (with $(\bar A_{jk})_{j, k=1, \dots, n}$ defined in~\eqref{avg_A})) and guarantee that $W_{\bar \G_n}$ converges to $W$ in the $L^{2}$-norm (corresponding to \eqref{e:beta} in~\S\ref{sss:condet}). Therefore, arguing exactly as  in \S\ref{sss:condet}, the only point we
  need to show is that
\begin{equation} \label{convbarg}
    \lim_{n \to + \infty} \| W - W_{\bar \G_n} \|_2=0\,.
\end{equation}
Now recall $\bar \G_n$  is the averaged graph of the scaled sparse $W$-random graph defined through \eqref{e:adjr1} and based on the trimmed graphon with target edge density $n^{\alpha}$. We now introduce an auxiliary graph, denoted by $\widetilde \G_n$, defined through the adjacency matrix as in ~\eqref{e:adj}, i.e., the $L^{2}$-orthogonal projection of $W$ on the space of piecewise constant functions on $[0,1]^2$. 
Then, considering the associated step-graphons, the standard triangle inequality gives
$$
    \| W - W_{\bar \G_n} \|_{2} \leq \| W - W_{ \widetilde \G_n} \|_{2}
+ \| W_{\widetilde \G_n} - W_{\bar \G_n}  \|_{2}\;.
$$
 Owing to estimate~\eqref{e:convl2} from Lemma~\ref{l:average}, the first term in the above sum vanishes in the $n \to + \infty$ limit. To control the limit of the second term, Jensen's inequality directly gives
\begin{equation*}
\begin{split}
     \| W_{\widetilde \G_n}& - W_{\bar \G_n}  \|_{2}  =
    \sum_{j, k=1}^n
    \f{ 
      |< W>_{I^n_j \times I^n_k}  - n^\alpha  < \widehat W^n >_{I^n_j \times I^n_k}|^2}{n^2}
  \\ & \stackrel{\text{Jensen}}{\leq} \!
     \sum_{j, k=1}^n \iint_{I^n_j \times I^n_k}\big[ W - \min\{n^\alpha, W\} \big]^2(x, y) dx dy 
    \! = \! \iint_{[0, 1]^2}\big[ W - \min\{n^\alpha, W\} \big]^2(x, y) dx dy\,.
\end{split}
\end{equation*}
As $\lim_{n \to + \infty} \big[ W - \min\{n^\alpha, W\} \big]^2(x, y)=0$ for a.e. $(x, y) \in [0, 1]^2$ and $|W - \min\{n^\alpha, W\}| \leq 2 W$ on $[0,1]^2$, then Lebesgue's Dominated Convergence Theorem yields~\eqref{convbarg} and this concludes the proof of Theorem \ref{thm:ran} {\sc (Case 1)}. 

\subsection{Proof of Theorem \ref{thm:ran} {\sc (Case 2)}} \label{ss:pcran}
In this subsection we assume $W\in L^{1}([0,1];\mathbb{R}_{+})$ to satisfy \eqref{bound_W2}.
The proof in this case follows the same lines as that of {\sc Case 1} above, so we only provide a sketch focusing on the points where the arguments are different.
We first show that 
\begin{equation} \label{e:diff3}
	\lim_{n\to+\infty} \Big[ \| [s^n  - \bar s^n](t, \cdot) \|_{L^1(0, 1)}
     + \| [i^n  - \bar i^n](t, \cdot) \|_{L^1(0, 1)}+ \| [r^n  - \bar r^n](t, \cdot) \|_{L^1(0, 1)} \Big] \! \! = \!0
	\end{equation}
$\mathbb P$-almost surely and for a.e. $t \in [0, T]$.
To this end, we subtract~\eqref{SIR_G} from~\eqref{SIR_avg}, multiply the result by $\mathrm{sign}(\bar s^n_j - s^n_j)$ and get 
\begin{equation*}
\begin{split}
\f{d}{dt} |\bar s^n_j - s^n_j| & \leq - \underbrace{\f{|\bar s^n_j - s^n_j|}{n} \sum_{k=1}^n \beta^n_k \bar A^n_{jk} \bar i_k^n}_{\ge 0}  
+  \f{|s_j^n| }{n} \sum_{k=1}^n  \beta^n_k \bar A^n_{jk} |\bar i^n_k  - i^n_k|
+
 \f{\mathrm{sign}[\bar s^n_j - s^n_j ]s_j }{n} \sum_{k=1}^n  \beta^n_k i^n_k  (\bar A^n_{jk}  - A^n_{jk}) 
    \\ & \stackrel{\eqref{ran:bd},\eqref{linftybk}}{\leq} 
\f{ \| \beta \|_{\infty}}{n} \sum_{k=1}^n  \bar A^n_{jk} | \bar i_k^n - i_k^n|+ H^n_j\,,
\end{split}
\end{equation*}
where we have set 
$$
   H^n_j (t): = \f{\mathrm{sign}[\bar s^n_j - s^n_j ]s_j^n }{n} \sum_{k=1}^n  \beta^n_k i^n_k  (\bar A^n_{jk}  - A^n_{jk})\,.
$$
This implies
\begin{equation}\label{sommaj}
\begin{split}
&\f{d}{dt}\sum_{j=1}^n |\bar s_j^n - s_j^n| \leq 
\f{ \| \beta \|_{\infty}}{n} \sum_{k=1}^n  | \bar i_k^n - i_k^n|\sum_{j=1}^n  \bar A^n_{jk}+
   \sum_{j=1}^n H^n_j
   \stackrel{\eqref{e:grado}}{\leq}
 K_a \| \beta \|_{\infty} \sum_{k=1}^n  | \bar i_k^n - i_k^n| +    \sum_{j=1}^n H^n_j\,.
\end{split}
\end{equation}
We then set   $\bar v^n_j : = \bar s^n_j + \bar i^n_j$, $v^n_j : =  s^n_j +  i^n_j$ and point out that owing to~\eqref{derdif} we have 
\begin{equation*}
 \begin{split}
  \frac{d}{dt} |\bar v^n_j - v^n_j| & \leq  - \underbrace{\gamma^n_j |\bar v^n_j - v^n_j|}_{\ge 0}
          +  \gamma^n_j |\bar s^n_j - s^n_j| \,,
\end{split}
\end{equation*}
so that summing the above expression over $j$, recalling~\eqref{sommaj}, using the inequality $|\bar i^n_j - i^n_j| \leq |\bar v^n_j - v^n_j| + |\bar s^n_j - s^n_j|$,   the definition of the $L^1$ norm as in \eqref{e:norma2} and applying Gr\"onwall's Lemma we arrive at 
\begin{equation}
\label{grodegree}\begin{split}
        &    \| [s^n  - \bar s^n](t, \cdot) \|_{L^1(0, 1)} 
     + \| [i^n  - \bar i^n](t, \cdot) \|_{L^1(0, 1)} \leq  \f{1}{n}  \exp[\bar D t ]
     \int_0^T \exp[-\bar D \tau]   \sum_{j=1}^n H^n_j(\tau) d\tau  \\ & = \f{1}{n^2} \sum_{j, k=1}^n 
    (\bar A^n_{jk}  - A^n_{jk})   \exp[\bar D t ]  \int_0^T \exp[- \bar D \tau] 
     \mathrm{sign}[\bar s^n_j - s^n_j ] [s_j^n  \beta^n_k i^n_k] (\tau) d \tau
   \qquad \text{almost surely}. 
\end{split}
\end{equation}
In the above expression $\bar D$ is a suitable constant depending only on $\| \beta \|_{\infty}, \| \gamma\|_{\infty}$ and $K_a$. 
To control the right hand side of~\eqref{grodegree}, note that  
\begin{equation} \label{intlim}
  \left| \int_0^T \exp[- \bar D \tau] 
     \mathrm{sign}[\bar s^n_j - s^n_j ][s_j^n  \beta^n_k i^n_k] (\tau) d \tau \right|
   \stackrel{\eqref{ran:bd},\eqref{linftybk}}{\leq} \| \beta \|_{\infty}  \int_0^T \exp[- \bar D \tau] d \tau
   \qquad \text{almost surely}. 
\end{equation}
Also, 
$$
   \mathbb{E} \left( \sum_{j, k=1}^n 
    (\bar A^n_{jk}  - A^n_{jk} ) \right)  = \sum_{j, k=1}^n \mathbb{E}(
    \bar A^n_{jk}  - A^n_{jk} ) \stackrel{\eqref{media0}}{=}0
$$
and 
$$
    n^\alpha  [  < \widehat W^n >_{I^n_j \times I^n_k} -1]  \leq
     \bar A^n_{jk}  - A^n_{jk} \leq n^\alpha < \widehat W^n >_{I^n_j \times I^n_k} \;.
$$
Hence, Hoeffding's inequality~\eqref{e:hoeff} yields 
$$
    \mathbb{P} \left(  \f{1}{n^2}  \left|\sum_{j, k=1}^n 
   ( \bar A^n_{jk}  - A^n_{jk})  \right| \ge t  \right) \leq \exp\left( -2 \left[ \frac{n^2 t}{n^{1 + \alpha}}\right]^2\right).
$$
By choosing now $t = n^\eta$ with $\eta \in\, ]\alpha-1, 0[$ and applying Borel-Cantelli Lemma 
we obtain 
$$ 
    \lim_{n \to + \infty} \f{1}{n^2}  \left|\sum_{j, k=1}^n 
    (\bar A^n_{jk}  - A^n_{jk})  \right| \leq  \lim_{n \to + \infty} n^\eta =0
    \qquad\text{almost surely}. 
$$
Coupling with ~\eqref{intlim} implies that the right hand side of~\eqref{grodegree} vanishes almost surely in the $n \to + \infty$ limit, establishing~\eqref{e:diff3}. The rest of the proof of Theorem \ref{thm:ran} {\sc (Case 2)} follows mutatis mutandis as in the proof of Theorem \ref{thm:ran} {\sc (Case 1)} and is therefore omitted. 

\section*{Acknowledgments}
This work was motivated by a question posed by Giovanni Naldi, whom the authors wish to thank. They also wish to thank Giuseppe Patan\`e and all the colleagues of the IMATI workgroup on COVID-19
for interesting discussions, and Enrico Priola and Giuseppe Savar\'e for clarifying remarks on time-dependent random variables. S.D. is supported by the INdAM GNAMPA 2023 Project {\em Modelli nonlineari in presenza di interazioni puntuali} and by the PRIN 2022 Project E53D23005450006. L.V.S. is a member of the GNAMPA  group of INDAM, of the PRIN 2020 Project 20204NT8W4, PRIN 2022 Project 2022YXWSLR, PRIN 2022 PNRR Project P2022XJ9SX, and of the CNR FOE 2022 Project STRIVE. 
%
%
%
%
%

\appendix
\section{Notation}
\label{sec:notation}
For the reader's convenience we collect here the main notation used in the present paper. 
\begin{itemize}
\item $\rr_+: = [0, + \infty[$
\item a.e., for a.e $x$: almost everywhere, for almost every $x$, with respect to the standard Lebesgue measure
\item $\mathbbm{1}_E$: the characteristic function of the set $E$
\item $\| \cdot \|_\Box:$ the cut norm defined by~\eqref{cut1}
\item $\delta_\Box$: the cut metric defined by~\eqref{cut metric}
\item $W_{\G}$: a step-graphon associated to the undirected graph $\G$ as in~\eqref{G-->W_G}
\item $I^n_j:$: see~\eqref{e:ij}
\item $< u>_{I^n_j}, < U>_{I^n_j\times I^n_k}$: see~\eqref{e:average}
\item $\mathcal G(n, W)$: the sparse $W$-random graph as in \eqref{ran:infty}
\item $\mathcal{G}(n, W, n^{-\alpha})$: the sparse $W$-random graph as in~\eqref{e:revisione} 
\item $n^\alpha \mathcal{G}(n, W, n^{-\alpha})$:  the scaled sparse $W$-random graph as in \eqref{e:adjr1}
\item $K_d$: the constant in \eqref{bound_W}
\item $K_0$: the constant in \eqref{e:bdl2}
\item $K_1$: the constant in \eqref{bound_Wn}
\item $K_a$: the constant in \eqref{bound_W2}
\item $C^\infty_c (\Omega)$: the set of infinitely differentiable, compactly supported functions on the open set $\Omega$
\item $C^0([0, T])$: the Banach space of continuous functions defined on the interval $[0, T]$, endowed with the norm
$$
    \| u \|_{C^0}: = \max_{t \in [0, T]} |u(t)|. 
$$
\end{itemize}

\section{Graphons and converging graph sequences}
	\label{sec:graphons}
	In this appendix we briefly overview some notions and results in the theory of graphs and graphons that we need in the present paper. The exposition mainly follows the work~\cite{BCCZ19}. 
	\subsection{Graphons and  cut distance}
     \begin{definition} \label{d:graphon}
	A \emph{graphon} is a summable function $W:[0,1]^2\to\rr$ satisfying $W(x, y) = W(y, x)$ for a.e. $x,y\in[0,1]$. In the following we denote by $\W$ the set of graphons. 
	\end{definition}
	Given a weighted undirected graph $\G_n$ with $n$ vertices and adjacency matrix $(A_{jk})_{j, k=1, \dots, n}$, 
	we can define a step-graphon $W_{\G_n}$ associated to $\G_n$ by using formula~\eqref{G-->W_G} (see also Figure \ref{f:fig1}).  
	
As anticipated in the Introduction, the theory of graphons makes crucial use of the cut norm defined by \eqref{cut1}. The notion of cut norm was first introduced by Frieze and Kannan \cite{FK99}, and its key importance for graphons has been recently highlighted in a series of papers (see \cite{BCCZ18,BCCZ19,BCLSV06,BCLSV08,BCLSV12} and the references therein). In particular, as pointed out in~\cite{BCLSV08} the cut norm is intimately connected with the notion of \emph{left convergence} for graph sequences (see Remark \ref{rem:left} below for a very brief discussion of this point and \S\ref{ss:analy} for some comments about the analytic properties of the cut norm). 
	
	From the graph theory viewpoint, a serious drawback of the notion of cut norm is the following. Assume $\G$ and $\G'$ are the very same graph, but have different vertex labeling. Then in general the adjacency matrices of $\G$ and $\G'$ are not the same and hence owing to~\eqref{G-->W_G} and~\eqref{cut1} we have $\| W_\G - W_{\G'} \|_{\Box} \neq 0$. A way out this issue is the notion of \emph{cut distance}. To define it, we first introduce some notation.  First, we recall that a map $\phi:[0,1]\to[0,1]$ is  measure--preserving if, for every measurable set $A\subseteq[0,1]$, the set $\phi^{-1}(A)$ is also measurable and furthermore $\left|\phi^{-1}(A)\right|=\left|\phi(A)\right|$. Given a graphon $W$ and a measure preserving map $\phi$, we set $W^\phi(x,y):=W(\phi(x),\phi(y))$, for every $x,y\in[0,1]$. 
	\begin{definition}
	For every $U, W \in \W$, the cut distance $\delta_\Box$ between $U$ and $W$ is defined as
	\begin{equation}
		\label{cut metric}
		\delta_{\Box}(U,W):=\inf_{\phi}\|U-W^\phi\|_\Box
	\end{equation}
	where the infimum is taken over all measure preserving maps $\phi:[0,1]\to[0,1]$. 
	\end{definition}
	Note that, strictly speaking, $\delta_\Box$  is only a pseudometric and not a distance since the equality $\delta_{\Box}(U,W)=0$ does not imply that $U \equiv W$. As a matter of fact, $\delta_{\Box}(W,W^\phi)=0$, for every fixed $W\in\W$ and any given measure preserving map $\phi$. For this reason, we will always tacitly identify two graphons with zero cut distance.
     \begin{remark}[Cut distance and left convergence of graph sequences]
     	\label{rem:left}
     Let us briefly recall the link between the cut metric and the so called left convergence of graph sequences, redirecting to~\cite{BCLSV08} for an extended discussion.  Very loosely speaking, a sequence of graphs $\{ \G_n\}_{n \in \nn}$ is left convergent if the \emph{local structure} of the graphs somehow stabilizes in the $n\to + \infty$ limit, in the sense that the (suitably normalized) number of copies in $\G_n$ of any given finite subgraph tends to a limit value as $n\to+\infty$ (see~\cite{BCLSV08} for the rigorous definition). The connection between left convergence and cut distance is unveiled by \cite[Theorem 3.8]{BCLSV08}: if $\{ \G_n \}_{n \in \nn}$ is a sequence of undirected graphs with uniformly bounded adjacency matrices, then $\{ \G_n \}_{n \in \nn}$ is left convergent if and only if $\lim_{n \to + \infty} \delta_\Box (W_{\G_n}, W)=0$, for some bounded graphon $W$.
    \end{remark}
	\subsection{Analytic properties of the cut distance}\label{ss:analy}
	We refer to~\cite{Jan14} for a detailed discussion of the cut norm from the point of view of mathematical analysis. Since we have used it in the paper, we recall the relation from~\cite[equation (4.1)]{Jan14} 
	\begin{equation}
	\label{cut2}
	\|W\|_{\Box} \leq \sup_{\|f\|_\infty,\|g\|_\infty\leq1}\left|\int_{[0,1]^2}W(x,y)f(x)g(y)\,dxdy\right|
     \leq  4 \|W\|_{\Box}. 
     \end{equation}
	where the supremum is taken over all real--valued, measurable functions $f,g$ defined on $[0,1]$.
	
	It is also interesting to point out that convergence in the cut norm is in between strong and weak convergence in $L^1([0,1]^2)$. Indeed,  the cut norm induces a weaker topology than the strong one, see \cite[Section 8.3]{Lov12}. On the other hand, one can show  that convergence in the cut norm is strictly stronger than the weak convergence, see \cite[Appendix F]{Jan14}. Loosely speaking, the reason why convergence in the cut norm is strictly stronger than weak convergence is the following. Weak convergence in $L^1([0,1]^2)$ can be characterized as convergence tested against characteristic functions in the form $\mathbbm{1}_{S \times T}$, with $S,T\subseteq[0,1]$. Convergence in the cut norm requires that this convergence is uniform with respect to $S$ and $T$.     
	\subsection{Compactness results for dense and sparse graph sequences}
	We now quote~\cite[Theorem 2.13]{BCCZ19} (see also~\cite[Theorem 5.1]{LS07} for the case $p= \infty$).
	\begin{theorem}
		\label{prop:comp}
		Assume $1<p\leq\infty$ and $C>0$ and consider a sequence of graphons $\{ W_n \}_{n \in \nn}\subset\W$ such that $\| W _n\|_p \leq C$ for every $n \in \nn$. 
		Then there is $W \in \W$ such that, up to subsequences,  
		$$
		    \lim_{n \to + \infty} \delta_{\Box} (W_n, W)=0. 
		$$
			\end{theorem}
			The analogous of Theorem~\ref{prop:comp} for $p=1$ is false in general, but compactness of bounded $L^1$ balls can be recovered by 
			adding a uniform integrability assumption, see~\cite[Theorem C.7]{BCCZ19}. 		
	We now recall the following fundamental definitions.
	       \begin{definition}
	       A sequence of undirected graphs $\{ \G_n \}_{n \in \nn}$ is said to be dense or sparse if 
	       $$
	           \liminf_{n \to + \infty} \| W_{\G_n} \|_{1} > 0 \quad \text{or} \quad  \lim_{n \to + \infty} \| W_{\G_n} \|_{1} = 0, 
	       $$
	       respectively. In the previous expression, $W_{\G_n}$ is the step graphon defined as in~\eqref{G-->W_G}. 
	       \end{definition}
	       The quantity $\| W_{\G_n} \|_{1}$ essentially represents the edge density of the graph. According to the previous definition, if  $\G_n$ is a sequence of dense simple graphs then the number of edges is $O(n^2)$ as $n$ is large enough, whereas if $\G_n$ is a sequence of sparse simple graphs then the number of edges is $o(n^2)$ as $n\to+\infty$.
	       Given a sequence of dense graphs $\{ \G_n \}_{n \in \nn}$, by applying Theorem~\ref{prop:comp} and its analogue in the case $p=1$ \cite[Theorem C.7]{BCCZ19} to the sequence $\{ W_{\G_n} \}_{n \in \nn}$ we obtain useful compactness results. The same results apply to sequences of sparse graphs, but in this case the associated sequence of graphons converges to the uninformative limit $W\equiv0$.  As a matter of fact, however, most of the graph sequences relevant for real-world applications are sparse, see the introduction to~\cite{BCCZ19}. As pointed out in~\cite{BCCZ19}, a way to circumvent this obstruction and extract information on the asymptotic behavior of sparse sequences is to normalize $\G_n$ and consider the sequence $\{ \G_n/  \| W_{\G_n} \|_{1}  \}_{n \in \nn}$.  Before discussing a selection of the main results in~\cite{BCCZ19} we have to introduce some further notation.   
\begin{definition}[$(C,\eta)-L^p$ upper regular graph] 
	Fix $C, \eta >0$. Given $p \in\, ]1, +\infty[$ and a graph $\G$ with $|V(\G)|=n \ge \eta^{-1}$ vertices and weighted adjacency matrix $A=(A_{k\ell})_{k,\ell=1}^n$, we say that $\G$ is a $(C,\eta)-L^p$ upper regular graph if the following holds. Let $\mathcal{P}:=\left\{V_1,\,\dots,\,V_m\right\}$ be any partition of $V(\G)$ into disjoint sets such that $|V_j|\geq\eta n$ for every $j=1,\,\dots,\,m$. Then 
	 \begin{equation}\label{e:lpreg}
	 \sum_{i, j=1}^m \frac{|V_i| |V_j|}{n^2} \left|\sum_{k \in V_i, \ell \in V_j} \frac{A_{k \ell}}{|V_i| |V_j|}      \right|^p   \leq C^p\|W_\G\|^p_{1}.
	 \end{equation}
	 For $p=\infty$, the definition is analogous: it suffices to replace \eqref{e:lpreg} with 
	 \[
	 \max_{1\leq i,j\leq m}\left|\sum_{k \in V_i, \ell \in V_j} \frac{A_{k \ell}}{|V_i| |V_j|}\right|\leq C\|W_\G\|_{1}.
	 \]
\end{definition}
	Very loosely speaking, a graph $\G$ is $(C,\eta)-L^p$ upper regular if, whenever its vertices are divided into a certain number of groups none of which is too small, then averaging the edge weights with respect to the partition gives a weighted graph with bounded $L^p$ norm. Note furthermore that if $p=1$ then equation~\eqref{e:lpreg} is always satisfied with $C=1$, so the definition is only meaningful if $p>1$. We now quote~\cite[Definition 2.7]{BCCZ19}.
\begin{definition}
	\label{d:Cup}
	 Fix $p \in\, ]1, + \infty]$ and assume that $ \{\G_n\}_{n \in \nn}$ is a sequence of sparse graphs such that $|V(\G_n)|=n$. We say that $\{\G_n\}_{n \in \nn}$ is a $C$--upper $L^p$ regular sequence if, for every $\eta>0$, there is $n_0=n_0(\eta)$ so that $\G_n$ is $(C+\eta,\eta)-$ upper $L^p$ regular, for every $n\geq n_0$.
     \end{definition}
Note that \cite[Proposition A.1]{BCCZ19} implies that, if $ \{\G_n\}_{n \in \nn}$ is a $C$--upper $L^p$ regular sequence of simple graphs for some $p>1$, then  the average degree of $\G_n$ blows up as $n \to + \infty$. The importance of Definition \ref{d:Cup} is given by the next result \cite[Theorem 2.8, Proposition 2.10]{BCCZ19}. 
	 \begin{theorem}
	 	\label{prop:L^p conv}
	 	Fix $p \in\, ]1, + \infty]$ and $C>0$. Assume that $\{\G_n\}_{n\in \nn}$ is a $C$--upper $L^p$ regular sequence of graphs.  Then there is $W\in\W$ such that $\|W\|_{p}\leq C$ and up to subsequences
	 	\begin{equation} \label{e:cutnormconv}
	 	\lim_{n \to + \infty}\delta_{\Box}\left( \f{W_{\G_n}}{\|W_{\G_n} \|_{1}},W \right)=0.
	 	\end{equation}
Conversely, if  $\{\G_n\}_{n\in \nn}$ is a sequence of graphs such that for some $W \in \W$ with $\| W \|_{p} < + \infty$ equation~\eqref{e:cutnormconv} holds true, then $\{\G_n\}_{n\in \nn}$ is a $\|W\|_{p}$--upper $L^p$ regular sequence.
	 \end{theorem}
          In the case $p=1$ the compactness result is still true provided we replace the assumption that  $\{\G_n\}_{n\in \nn}$ is a $C$--upper $L^p$ regular sequence of graphs with the assumption that  $\{\G_n\}_{n\in \nn}$ is a uniformly upper regular sequence, see~\cite[Theorem C.13]{BCCZ19}. 
          
 \subsubsection{Convergence in the cut metric and convergence in the cut norm}
     Summing up, cut--metric convergence provides a convenient framework to discuss the asymptotic properties of sequences of both dense and sparse large graphs. As a matter of fact, however, one often consider sequences converging in the cut norm. Although this might sound like a stronger assumption, the following proposition states that the two conditions are equivalent, up to a relabeling of the vertices. For the proof we refer to \cite[Lemma 5.3]{BCLSV08} in the case $p=+\infty$ and \cite[Proposition 5.2]{BCCZ19} in the case $1<p<\infty$.
 	 \begin{proposition}
 	 	If $\{\G_n\}_{n \in \nn}$ is a sequence of graphs satisfying the assumptions of either Theorem \ref{prop:comp} or Theorem \ref{prop:L^p conv}, then there are $W\in \W$ and, for every $n$, a labeling of the vertices of $\G_n$ such that
 	 	\[
 	 	\|W_{\G_n}-W\|_{\Box}\to0\qquad\text{or}\qquad\left\|\f{W_{\G_n}}{\|W_{\G_n\|_1}}-W\right\|_{\Box}\to0
 	 	\]
         as $n \to + \infty$, respectively.  
 	 \end{proposition}
\begin{remark} \label{r:discussion}
	By combining Theorems~\ref{prop:comp}--\ref{prop:L^p conv} we can conclude that assumption~\eqref{e:conv_tloc_cut} is satisfied by a large class of sequences of undirected graphs. In particular, owing to Theorem~\ref{prop:comp} it is satisfied up to subsequences and vertex relabelings by any sequence $\{\G_n\}_{n\in \nn}$ such that $\| W_{\G_n} \|_{p}$ is uniformly bounded. If the sequence $\{\G_n\}_{n\in \nn}$ is dense, then the limit $W$ we obtain this way is already meaningful.  If the sequence  $\{\G_n\}_{n\in \nn}$ is sparse, then the limit is the uninformative trivial graphon $W \equiv 0$. However, Theorem~\ref{prop:L^p conv} implies that, in the sparse case,  assumption~\eqref{e:conv_tloc_cut} is satisfied  (up to subsequences and vertex relabelings) by the normalized sequence $\{\G_n/ \| W_{\G_n} \|_{1}\}_{n\in \nn}$ provided   $\{\G_n\}_{n\in \nn}$ is $C$--upper $L^p$ regular for some $C>0$ and $p>1$.  
\end{remark}

\end{document}